\documentclass[1p]{elsarticle}
\usepackage{latexsym,amssymb}
\usepackage[english]{babel}

\usepackage{lscape}

\usepackage{amsfonts, amsthm}

\def\Z{\mathbb{Z}}

\def\BHR{\mathop{\rm BHR}}

\newtheorem{defi}{Definition}[section]
\newtheorem{prop}[defi]{Proposition}
\newtheorem{lem}[defi]{Lemma}
\newtheorem{rem}[defi]{Remark}

\newtheorem{thm}[defi]{Theorem}
\newtheorem*{conj}{Conjecture}

\begin{document}
\begin{frontmatter}
\title{A new result on the
problem of Buratti, Horak and Rosa}

\author[ap]{Anita Pasotti\corref{cor1}}
\address[ap]{DICATAM - Sez. Matematica, Universit\`a degli Studi di Brescia, Via
Valotti 9, I-25133 Brescia, Italy}
\ead[ap]{anita.pasotti@ing.unibs.it}
\cortext[cor1]{Corresponding author}

\author[mp]{Marco Antonio Pellegrini}
\address[mp]{Departamento de Matem\'atica, Universidade de Bras\'ilia - ICC Centro,
70910-900  Bras\'ilia - DF, Brazil}
\ead[mp]{pellegrini@unb.br}

\begin{keyword}
Hamiltonian path \sep complete graph \sep edge-length
\MSC[2010]{05C38}
\end{keyword}

\begin{abstract}
The conjecture of Peter Horak and Alex Rosa (generalizing that of Marco Buratti) states that a multiset $L$ of $v-1$ positive integers
not exceeding $\lfloor {v\over2}\rfloor$ is the list of {\it edge-lengths} of a suitable Hamiltonian path of the complete
graph with vertex-set $\{0,1,\ldots,v-1\}$ if and only if the following condition (here reformulated in a slightly easier form)
is satisfied: for every divisor $d$ of $v$, the number of multiples of $d$ appearing in $L$
is at most $v-d$.
In this paper we do some preliminary discussions on the conjecture, including its relationship
with graph decompositions. Then we prove, as main result, that the conjecture is true
whenever all the elements of $L$ are in $\{1,2,3,5\}$.
\end{abstract}
\end{frontmatter}

\section{Introduction}
Throughout this paper $K_v$ will denote the complete graph on $\{0,1,\dots,v-1\}$ for any positive integer $v$.
For the basic terminology on graphs we refer to \cite{Wbook}.
Following \cite{HR}, we define the {\it length} $\ell(x,y)$ of an edge $[x,y]$ of $K_v$ as
$$\ell(x,y)=min(|x-y|,v-|x-y|).$$
If $\Gamma$ is any subgraph of $K_v$, then the list of edge-lengths of $\Gamma$ is the multiset
$\ell(\Gamma)$ of the lengths (taken with their respective multiplicities) of  all the edges
of $\Gamma$. For our convenience, if a list $L$ consists of
$a_1$ $1'$s, $a_2$ $2'$s, \ldots, $a_t$ $t'$s,
we will write $L=\{1^{a_1},2^{a_2},\ldots,t^{a_t}\}$.

The following conjecture \cite{W} is due to Marco Buratti
(2007, communication to Alex Rosa).

\begin{conj}[Buratti]
For any prime $p=2n+1$
and any multiset $L$ of $2n$ positive integers not exceeding $n$, there exists a
Hamiltonian path $H$ of $K_p$ with $\ell(H)=L$.
\end{conj}

Buratti himself never worked at his conjecture but he
finally mentions it in \cite[p. 14]{BM}. The problem appears to be very difficult, so much so that Alex Rosa
defined it a {\it combinatorial disease} in his lecture at the international conference {\it Combinatorics 2008}, held in
Costermano (Italy) from 22 June to 28 June 2008. The conjecture is almost trivially true
in the case that $L$ has just one edge-length. The case of exactly two distinct edge-lengths
 has been solved independently in \cite{DJ, HR} but, for the time being, the case of exactly three distinct edge-lengths has been solved
only when these lengths are 1, 2 and 3, see \cite{CDF}.
The authors of the present paper are investigating in \cite{PPnew} the more general case
 in which the elements of the list are 1, 2 and $t$, where $t$ is an arbitrary positive integer.
Mariusz Meszka checked that the conjecture is true for all primes $p\leq23$ by computer.  Some  general results on the conjecture can be found in \cite{HR}, in particular that it is true when there is an edge-length occurring ``sufficiently many times'' in $L$.

In \cite{HR} Peter Horak and Alex Rosa generalized Buratti's conjecture as follows.
\begin{conj}[Horak and Rosa]\label{Conj HR}
Let $L$ be a list of $v-1$ positive integers not exceeding $\lfloor{v\over2}\rfloor$.
Then there exists a Hamiltonian path $H$ of $K_v$ such that $\ell(H)=L$ if, and only if, the following condition holds:
\begin{equation}\label{HR}
\left. \begin{array}{c}
\textrm{for any sublist $J$ of $L$ with $J \ \cap \ (L\setminus J)=\emptyset$,}  \\
\textrm{we have $|J|\geq \gcd\{v,\ell \ | \ \ell \in L\setminus J\}-1$.}
\end{array}\right.
%\end{split}
\end{equation}
\end{conj}

Note that for a sublist $J$ of $L$, we may actually have
$J \ \cap \ (L\setminus J)\neq\emptyset$ since $L$ and $J$ are multisets; for instance,
if $L=\{1,1,1,1,2,2,2\}$ and $J=\{1,2,2\}$, then we have $L\setminus J=\{1,1,1,2\}$ and
hence $J \ \cap \ (L\setminus J)=\{1,2\}$.

It is evident that condition (1) is always trivially satisfied when $v$ is a prime. Therefore in this case
the conjecture of Horak and Rosa reduces to that of Buratti.
For short, by $\BHR(L)$ we will mean the above conjecture
for a given list $L$.
A positive answer
of $\BHR(\{\ell_1^a,\ell_2^b\})$ was presented in \cite{HR}, namely
it was proved that the conjecture is true when $L$ has exactly two distinct edge-lengths.
We also observe that the result of Capparelli and Del Fra \cite{CDF} allows to see
that the conjecture of Horak and Rosa (not only that of Buratti) is true when all the
elements of $L$ are in $\{1,2,3\}$, so $\BHR(\{1^a,2^b,3^c\})$ holds for all integers
$a,b,c\geq0$.

\medskip
Next proposition gives an equivalent easier form to state condition (\ref{HR}).
\begin{prop}\label{HR iff B}
Condition $(\ref{HR})$ is equivalent to the following:
\begin{equation}\label{B}
\left. \begin{array}{c}
\textrm{for any divisor $d$ of $v$, the number of multiples of $d$} \\
\textrm{appearing in $L$ does not exceed $v-d$.}
\end{array}\right.
\end{equation}
\end{prop}
\begin{proof}
Assume that (\ref{HR}) holds and let $d$ be a divisor of $v$. Consider the sublist $J$ of $L$ consisting of all elements of $L$ not divisible by $d$.
So $J \ \cap \ (L\setminus J)$ is empty and  $L\setminus J$ is the sublist of all elements of $L$ divisible by $d$. Thus $d$ is a divisor of
$\gcd\{v,\ell \ | \ \ell \in L\setminus J\}$ and then, by (\ref{HR}), we have $|J|\geq d-1$. This implies $|L\setminus J|=v-1-|J|\leq v-d$, i.e., (\ref{B}) holds.

Conversely, assume that (\ref{B}) holds, let $J$ be a sublist of $L$ with $J \ \cap \ (L\setminus J)=\emptyset$, and set $d=\gcd\{v,\ell \ | \ \ell \in L\setminus J\}$.
Of course every element of $L\setminus J$ is a multiple of $d$ which is a divisor of $v$. Then, by (\ref{B}), we have $|L\setminus J|\leq v-d$.
It follows that $|J|=v-1-|L\setminus J|\geq d-1$, i.e., (\ref{HR}) holds.
\end{proof}

We remark that Seamone and Stevens gave another condition which is similar to (\ref{B}) but not completely correct.
They in fact just require that ``for all divisors $d$ of $v$ there are no more than $v-d$ elements
$\ell$ of $L$ such that $\gcd(\ell,v)=d$" (see \cite{SS}, Conjecture 1.1) which is not equivalent to (\ref{HR}).
For instance, if $v=8$
and $L=\{2,2,2,4,4,4,4\}$, we see that condition (2) (hence, equivalently, (1))
does not hold but that  the condition of Seamone and Stevens does.
It is also evident that no Hamiltonian path of $K_8$ may have $L$ as its list of edge-lengths.

We also point out that in the statement of the conjecture of Horak and Rosa, the actual conjecture is the sufficiency. Indeed Horak and Rosa themselves
remarked that its necessity is certainly true but they did not prove it explicitly. Here, for
convenience of the reader, we give a proof in terms of condition (2).

\begin{prop}
The list $L$ of edge-lengths of any Hamiltonian path of $K_v$ satisfies condition $(\ref{B})$.
\end{prop}
\begin{proof}
Let $H$ be a Hamiltonian path of $K_v$ with $\ell(H)=L$ and let $d$ be a divisor of $v$.
Denote by $N$ the number of non-multiples of $d$ appearing in $L$ and note that
condition (2) is equivalent to saying that $N\geq d-1$.
To prove that this inequality holds, we consider the graph $\Gamma$
obtainable from $H$ by deleting all $N$ edges whose length is not divisible by $d$.
It is easy to see that $\Gamma$ has exactly $N+1$ connected components
some of which may be just isolated vertices.
It is also clear that all vertices of every connected component $C$ of $\Gamma$
are in the same residue class modulo $d$ so that $C$ has
at most $\frac{v}{d}$ vertices, since every class of residues modulo $d$ intersects $\{0,\ldots,v-1\}$ in a set of that size.
It follows that $v=|V(\Gamma)|\leq (N+1)\frac{v}{d}$, i.e. $N\geq d-1$.
\end{proof}

Before giving the main result of this paper we would like to show some
connections between $\BHR$-problem and \emph{graph
decompositions} (see \cite{BE} for general background on this subject).

A list $\Omega$ of elements of an additive group $G$ is
said to be \emph{symmetric} if $0 \notin \Omega$ and the multiplicities
of $g$ and $-g$ in $\Omega$ coincide for any $g \in G$.
If $\Omega$ does not have repeated elements then one can consider the \emph{Cayley
graph on $G$ with connection set $\Omega$}, denoted $Cay[G:\Omega]$, whose
vertex-set is $G$  and in which $[x,y]$ is an edge if and only if $x-y \in \Omega$.
Cayley graphs have a great importance in combinatorics
and they are precisely the graphs admitting an automorphism
group acting sharply transitively on the vertex-set (see, e.g., \cite{GR}).
If, more generally, the symmetric list $\Omega$ has repeated
elements one can consider the \emph{Cayley multigraph on $G$ with
connection multiset $\Omega$}, also denoted $Cay[G:\Omega]$
and with vertex-set $G$, where the multiplicity of an edge $[x,y]$ is
the  multiplicity of $x-y$ in $\Omega$ (see, e.g., \cite{BCD, BMnew}).
Note that any Cayley (multi)graph
 is regular of degree the size of its connection (multi)set.

 Now recall that the \emph{list of differences} of a simple graph $\Gamma$
with vertices in an additive group $G$ is the symmetric multiset $\Delta(\Gamma)$
of all possible differences $x-y$ with $(x,y)$ an ordered pair of adjacent
vertices of $\Gamma$ (see, e.g., \cite{BP}).

We point out that Buratti communicated his conjecture to Rosa
using the above terminology:
``For any prime $p$ and any symmetric list $\Lambda$ of $2p-2$ elements
of $\Z_p$, there exists a Hamiltonian path $H$ of $K_p$ such that
$\Delta(H) = \Lambda$''. It is clear that the conjecture of Horak and Rosa can be
also reformulated is a similar way.

The reader who is familiar with graph decompositions with a {\it regular automorphism group}
can easily recognize that if $H$ is a Hamiltonian path of
the complete graph with vertex-set an additive group $G$,
then the collection $\{H+g\ |\ g\in G\}$ is a decomposition of
$Cay[G:\Delta(H)]$ into Hamiltonian paths ($H+g$
denotes the path obtainable from $H$ by replacing each
$x\in V(H)$ with $x+g$).
This observation allows us to reformulate Buratti's conjecture
in the following form:
``Any Cayley multigraph of order a prime number $p$ and degree
$2p-2$ admits a cyclic decomposition into Hamiltonian paths''.
The more general $\BHR(L)$ can be reformulated as follows.

\begin{conj} A Cayley multigraph {\it $Cay[\Z_v:\Lambda]$ admits a cyclic decomposition into Hamiltonian paths if and only if $\Lambda=L \ \cup \ -L$
with $L$ satisfying condition $(\ref{B})$}.\end{conj}

The following more general problem
can be considered:
``Given a simple Cayley graph $Cay[G:\Omega]$ and given
another graph $\Gamma$, determine all symmetric lists $\Lambda$ of
elements of $G$ for which there exists a subgraph $\Gamma'$
of $Cay[G:\Omega]$ isomorphic to $\Gamma$ such that
$\Delta(\Gamma') = \Lambda$''.
Some results in this direction have been presented by Seamone and Stevens \cite{SS}.\\

\noindent The main result of the present paper is the following.
\begin{thm}\label{Bur-Rosa}
$\BHR(\{1^a,2^b,3^c,5^d\})$ holds for all integers $a,b,c,d\geq0$.
\end{thm}
\noindent So, in particular, we have found the first set $S$ of size four for which we can say that
$\BHR(L)$ is true when the underlying-set of the list $L$ is $S$.

In Section 2 we will introduce \emph{cyclic} and \emph{linear realizations} which are the fundamental
tools we have used to obtain our result. Also, for the convenience of the reader, we will explain the strategy
which allows us to get Theorem \ref{Bur-Rosa}, starting from the results contained in \cite{CDF},
and how this strategy can be applied in a general case.
In this way the reader will be able to understand why we have chosen  $S=\{1,2,3,5\}$
instead of the (maybe) more natural choice $S=\{1,2,3,4\}$.
In Section 3 we will present all the constructions of the linear realizations and Section
4 is dedicated to the proof of Theorem \ref{Bur-Rosa}.
Finally, in Section 5 we will make some final remarks about possible future results on
BHR-problem.

\section{Cyclic and linear realizations}
In this section we are going to introduce cyclic and linear realizations
of a list $L$
and to explain their relationship with $\BHR(L)$.\\
A \emph{cyclic realization} of a list $L$ with $v-1$ elements each
from the set
$\{1,\ldots,\lfloor\frac{v}{2}\rfloor\}$ is a Hamiltonian path $[x_0,x_1,\ldots,x_{v-1}]$
of $K_{v}$ such that the multiset of edge-lengths
$\{\ell(x_i,x_{i+1})\ |\ i=0,\ldots,v-2\}$ equals $L$. So it is clear that
$\BHR(L)$ can be so reformulated:
 every such a list $L$ has a cyclic realization
if and only if condition (\ref{B}) is satisfied.
For example the path $[0,3,8,2,5,10,4,9,1,6,7]$ is a cyclic realization of
$L=\{1,3^3,5^6\}$.

In order to investigate $\BHR$-problem, it is useful to introduce also
(perfect) linear realizations, see \cite{CDF}.
A \emph{linear realization} of a list $L$ with $v-1$ positive integers not exceeding
$v-1$ is a Hamiltonian path
$[x_0,x_1,\ldots,x_{v-1}]$ of $K_v$
such that $L=\{|x_i-x_{i+1}|\ |\ i=0,\ldots,v-2\}$.
For instance, one can easily check that the path $[0,5,3,1,6,9,4,2,7,12,10,11,8]$ is a
linear realization
of $L=\{1,2^4,3^2,5^5\}$.

In this paper we shall choose $0$ as first vertex of any path.

\begin{rem}\label{cyclin}
Every linear realization of a list $L$ can be viewed as a
cyclic realization
of a suitable list $L'$ but not necessarily of the same list. For example, the path
$[0,2,3,5,4,1,6]$ is a linear realization of
$L=\{1^2,2^2,3,5\}$ and a cyclic realization of $L'=\{1^2,2^3,3\}$.
Anyway if all the elements in the list are less than or equal to $\lfloor\frac{|L|+1}{2}\rfloor$, then every linear realization of $L$ is
also a cyclic realization of the same list $L$ $($see Section $3$ of \cite{HR}$)$.
\end{rem}

We remark that in order to prove $\BHR(L)$ we have to find cyclic realizations
of $L$, nevertheless we focus on linear
realizations
because they will be used in an inductive construction, as we are going to explain.

Following \cite{CDF}, we will say that a linear realization of a list $L$ is \emph{perfect},
and we denote it by $RL$,
if the terminal vertex of the path is labelled by the largest element. We denote by $rL$
a linear realization which may or may not be perfect. Given a perfect realization
$RL_1=[0,x_1,\ldots,x_{s-1},s]$
and another realization
$rL_2=[0,y_1,\ldots,y_t]$, not necessarily perfect, we may
form a new realization
$r(L_1\cup L_2)$ denoted by $RL_1+rL_2$ so defined
$$(*)\quad\quad RL_1+rL_2=[0,x_1,\ldots,x_{s-1},s,y_1+s,\ldots,y_t+s].$$
It is important to underline that the previous construction, in general, does not work if
we consider
\emph{cyclic} realizations.\\
As we shall see, this construction plays a fundamental role in this paper.\\

As we have already remarked, to provide a complete proof of the conjecture seems to be
very difficult.
Here we present the strategy we have used to solve BHR$(\{1^a,2^b,3^c,5^d\})$ and that can
be applied to any given list. Our starting point is the following result about linear
realizations, obtained by Capparelli and Del Fra in
\cite{CDF}.

\begin{thm}{\rm\cite{CDF}} \label{KDF}
A multiset $\{1^a,2^b,3^c\}$ has a linear realization if, and only if, the integers $a,b,c$
satisfy one of the following conditions
\begin{enumerate}
\item[\rm{(i)}] $a=0$, $b\geq 4$, $c\geq 3 $;
\item[\rm{(ii})] $a=0$, $b=3$ and $c\neq0, 3k+9$ with $k\geq0$;
\item[\rm{(iii)}] $a=0$ and $(b,c)\in \{(2,2),(2,3),(4,1),(4,2),(7,2),(8,2)\}$;
\item[\rm{(iv)}] $a\geq 2$ and  $b=0$;
\item[\rm{(v)}] $a\geq 1$ and $c=0$;
\item[\rm{(vi)}] $a,b,c\geq 1$ with $(a,b,c)\neq (1,1,3k+5)$ and $k\geq 0$.
\end{enumerate}
\end{thm}

Next, we observe that if $L$ has only one symbol $s$ then it admits no
linear realization unless $s=1$ and
in this case the trivial perfect realization is $R\{1^a\}=[0,1,2,\ldots,a]$. So in the
following we assume
$L=\{1^a,2^b,3^c,5^d\}$
with at least two integers in $\{a,b,c,d\}$ greater than or equal to $1$. Also, in view of
Theorem \ref{KDF}, we can
assume $d\geq 1$.
\begin{rem}\label{RemA>a}
If there exists a linear realization of $L=\{1^a,2^b,3^c,5^d\}$, then
$R\{1^{A-a}\}$ $+rL$ is a
linear realization of
$L'=\{1^A,2^b,3^c,5^d\}$ for any $A\geq a$.
\end{rem}

\begin{rem}\label{a+b+c}
If $a+b+c<4$ and $d\geq1$, the list $L=\{1^a,2^b,3^c,5^d\}$ cannot admit a linear
realization,
since it is not possible
to obtain all congruence classes modulo $5$.
\end{rem}

Thanks to the strategy we are going to explain, we obtained the following theorems.

\begin{thm}\label{lin a=0}
Assume $b\geq 3$ and $c,d\geq 1$. Then the
list $L=\{2^b,3^c,5^d\} $ has a linear realization
if, and only if,
$(b,c,d)\neq (3,1,5k+8)$, $(\bar b,1,1)$ with $k\geq 0$ and either $\bar b=7,8$ or
$\bar b\geq 11$.
\end{thm}

\begin{thm}\label{MainResult}
Assume $a,d\geq 1$. Then the
list $L=\{1^a,2^b,3^c,5^d\} $ always has a linear realization except
when $(a,b,c,d)$ has one of the following forms:
\begin{enumerate}
\item[\rm{(i)}] $(1,0,c,1)$ with either $c\leq 3$ or $c\geq 8$;
\item[\rm{(ii)}] $(1,0,c,2)$ with either $c\leq 2$ or $c\geq 7$;
\item[\rm{(iii)}] $(1,0,0,d)$, $(1,0,1,d)$, $(1,0,2,d)$, $(1,1,0,d)$,
$(1,1,1,d)$, $(1,2,0,d)$, $(2,0,0,d)$, $(2,0,1,d)$, $(2,1,0,d)$, $(3,0,0,d)$;
\item[\rm{(iv)}] $(1,0,3,5k+6)$, $(1,0,3,5k+8)$, $(1,0,3,5k+10)$,
$(1,1,2,5k+7)$, $(1,1,2,5k+9)$, $(1,2,1,5k+8)$,
$(1,3,0,5k+7)$, $(1,3,0,5k+9)$, with $k\geq 0$;
\item[\rm{(v)}] $(3,1,0,5k+8)$, $(4,0,0,5k+8)$, with $k\geq 0$.
\end{enumerate}

\end{thm}

Note that any list $L$ with elements in a set $S$ can be identified with the function from $S$
to the set $\mathbb{N}$ of the natural numbers,
sending every $s\in S$ in the number of times that $s$ appears in $L$.
Thus it is natural to denote by
 $\mathbb{N}^S$ the set of all possible lists whose
elements are in $S$ and by $\mathcal{R}(\mathbb{N}^S)$ the set of all linearly realizable lists of $\mathbb{N}^S$.
In the following let $S=\{1,2,3\}$ and note that, in this case,
$\mathcal{R}(\mathbb{N}^S)$ is described in Theorem \ref{KDF}.\\
Now we are ready to present how we obtained our results.

\medskip
\noindent \textbf{Step 1.} In view of Remark \ref{RemA>a} we start investigating the
linear realizations of $L=\{2^b,3^c,5^d\}$.
Clearly, by Remark \ref{a+b+c} this means $b+c\geq4$. In addition, here we suppose
$b\geq3$.
First, we construct some perfect linear realizations of $L_1=\{2^b,3^c,5^d\}$ with $b,c$
as small as possible, see Table \ref{TablePerfect}.
Then, we apply the composition $(*)$ where $RL_1$ is one of these new perfect linear
realizations and
$L_2$ is in $\mathcal{R}(\mathbb{N}^S)$.
In such a way we are left to consider lists of the two following types:
\begin{itemize}
\item[(i)] $L=L_1\cup L'_2$ with $L'_2\in S^\mathbb{N}\setminus\mathcal{R}(\mathbb{N}^S)$;
\item[(ii)] $L=L'_1\cup L_2$ with  $L_2\in\mathcal{R}(\mathbb{N}^S)$, but $L'_1$ not containing any $L_1$ admitting a perfect linear realization.
\end{itemize}
%Then we apply Remark \ref{RemA>a}
%to each of these linear realizations.
%In such a way we solve BHR$(2^b,3^c,5^d)$ for any $b\geq4$, $c\geq3$ and $d\geq0$.\\
%Now it remains to construct a realization for the lists $L'=L_1\cup L_2$ for which $rL_2$ does not exists
%and for the lists $L''=\{1^\bar,2^\bar b,3^\bar c,5\bar d\}$ with at least one $\bar b,\bar c, \bar d$ less than the above $b,c,d$.\\

\medskip
\noindent\textbf{Step 2.} When possible, we find a linear realization for the lists
$L=\{2^b,3^c,5^d\}$
not covered by Step 1. For this reason we construct some linear realizations of $\bar L=\{2^b,3^c,5^d\}$
with $b,c$ as small as possible, see Table \ref{TableLinear}. Now we apply the
composition $(*)$,
where $RL_1$ is a perfect linear realization of Table \ref{TablePerfect} (in few cases we use a perfect linear realization constructed in \cite{CDF})
and $r \bar L$ is a linear realization of Table \ref{TableLinear}.
To conclude we divide the remaining realizations into 129 infinite families 
in which exactly one among $b,c,d$ is not a given integer, but depends on a parameter 
and it is written 
according to
the congruence class of $b$ modulo 2, or of $c$ modulo $3$,
or of $d$ modulo $5$, respectively 
(for example $L=(2^{2v+3},3^2,5^1)$, $L=(2^{4},3^{3t+1},5^1)$ or $L=(2^{5},3^{0},5^{5k+1})$). 
Then, for each of these families, we directly construct a linear
realization, see \cite{PP}.

\medskip
\noindent \textbf{Step 3.} Now we consider the case $a=1$. In view of Remark \ref{RemA>a}
we have  to
find a linear realization of the list $L=\{1,2^b,3^c,5^d\}$ only if we have not found a linear realization
of the list $L'=\{2^b,3^c,5^d\}$ in Step 1 or Step 2 (see Theorem \ref{lin a=0}).
Also here we proceed by using the composition $(*)$, and, when this is not possible, by a
direct construction, see \cite{PP}.

\medskip \noindent \textbf{Step 4.} For $2\leq a\leq 5$, again in view of Remark
\ref{RemA>a},
we have to construct a linear realization of the list $L=\{1^a,2^b,3^c,5^d\}$ only if we have not found a linear realization
of the list $L'=\{1^{a-1},2^b,3^c,5^d\}$.
As done before, we proceed by using the composition $(*)$ and, when this is not possible,
by a direct construction, see \cite{PP}.
Finally, for $a\geq 6$ the result follows by Remark \ref{RemA>a} and the previous
constructions.

\medskip \noindent \textbf{Step 5.} As final step, we construct a cyclic realization for
the lists which do not admit a linear realization, namely when $a+b+c<4$ (see Remark \ref{a+b+c}),
and for the lists for which
we have not found a linear realization, namely when $a=0$ and $b=1,2$ (see Theorem
\ref{MainResult}).
We subdivide these lists into $86$ infinite families according to congruence class of $c$
modulo $3$ or of $d$ modulo $5$
and in Section 4 we give a direct construction of a cyclic realization for each of
these families.\\

The choice of $S=\{1,2,3,5\}$ instead of the more natural one of $\{1,2,3,4\}$ is due to
the
strategy illustrated above, since a perfect linear realization of $L=\{1^a,2^b,3^c,5^d\}$
exists for many small values
of $a,b,c$. So this means that, for instance, in Step 1 the number of lists of type (ii)
is definitively smaller if we consider $\{1,2,3,5\}$
instead of $\{1,2,3,4\}$. Also in the other steps, with our choice it is possible to
obtain more realizations thanks to the composition construction, and this reduces significantly
the direct constructions one has to provide.

%For instance, for infinite choices of $d$, we can use $4$ integers in
%$\{1,2,3\}$ to separate the $5$ classes of $|L|+1$ module $5$, obtaining a perfect linear
%realization.
%On the other side, it is possible to use $3$ integers
%in $\{1,2,3\}$ to
%separate the $4$ classes of $|L|+1$ module $4$, obtaining a perfect linear realization,
%only when the list has few $4$.\\
%
% While it
% is not possible to construct a perfect linear realization of $L=\{1^a,2^b,3^c,4^d\}$
% with $a+b+c=4$,
In order to give an idea to the reader of the reason
for which
it is important the parity of the largest element of $L$ consider, for instance, the list
$L=\{1^{m-1},m^{km}\}$ where $m$ is an odd integer and $k$ is an arbitrary
positive integer.
It is not hard to see that the following is a perfect linear realization of $L$:
$RL=[0,m,2m,\ldots,km,km+1,(k-1)m+1,\ldots, 1, 2,m+2,2m+2\ldots, km+2,km+3,(k-1)m+3,\ldots,3,4,m+4,2m+4,\ldots,km+4,\ldots, m-1,2m-1,\ldots,
km+m-1 ]$.
\\
Suppose now $m$ even, for instance take $m=4$. It is easy to see that a perfect linear
realization of
$L'=\{1^3,4^{4k}\}$ cannot exist. In fact the unique linear realization of $L'$ is
$rL'=[0,4,8,\ldots,4k,4k+1, 4(k-1)+1,\ldots, 1, 2,6, 10,\ldots, 4k+2,4k+3, 4(k-1)+3,\ldots,
7,3]$,
which clearly is not perfect.

\section{Construction of linear realizations}
In this section we construct linear realizations for the lists $L=\{1^a,2^b,3^c,5^d\}$.
As already observed, we can assume $d>0$ and at least one of $a,b,c$ grater than $0$.

%\begin{rem}\label{RemA>a}
%If there exists a linear realization of $L=\{1^a,2^b,3^c,5^d\}$, then
%$R\{1^{A-a}\}$ $+rL$ is a
%linear realization of
%$L'=\{1^A,2^b,3^c,5^d\}$ for any $A\geq a$.
%\end{rem}
%
%\begin{rem}\label{a+b+c}
%If $a+b+c<4$ and $d\geq1$, the list $L=\{1^a,2^b,3^c,5^d\}$ cannot admit a linear realization,
%since it is not possible
%to obtain all congruence classes modulo $5$.
%\end{rem}

In order to deal with the general case, we need some \emph{perfect} linear
realizations that we
list in Table \ref{TablePerfect}. We also need the linear realizations of
Table \ref{TableLinear}. For certain values of $(a,b,c,d)$ it is a simple exercise to obtain a
linear realization of $L$. For instance, we have found
a perfect linear realization of the lists $\{3^4,5^4\}$,  $\{2^2,3^2,5^3\}$,
$\{2^2,3^3,5\}$ and $\{2^6,5\}$ that will be often used in the following
constructions.
Also, in some cases  we directly construct a linear realization of a given list $L$.
For sake of brevity  we have collected all these constructions
in \cite{PP}.

In order to represent our realizations in a short way we will always use the notation here explained.
In every realization of a list $L=\{1^{a_1},2^{a_2},3^{a_3},5^{a_5}\}$ in which exactly
one $a_i$ is not a given integer but depends on a parameter, the dots always appear between two elements
$x$ and $y$ which are congruent modulo $i$. It is then understood that the sequence of elements
between $x$ and $y$ is the arithmetic progression $x+i,x+2i,\ldots,y-i$ or the arithmetic progression
$x-i,x-2i,\ldots,y+i$ according to whether $y>x$ or $y<x$, respectively.

\begin{center}
\begin{table}[!th]
\begin{footnotesize}
\begin{tabular}{ll}
$(a,b,c,d)$ & $R\{1^a,2^b,3^c,5^d\}$\\\hline
$(0,0,6,5k+2)$ & $[0,3,\ldots,5k+3,5k+6,\ldots,1 , 4,\ldots, 5k+4
, 5k+7,\ldots, 2,5,\ldots,5k+5, 5k+8]$\\

$(0,0,6,5k+3)$ & $[0,3,\ldots,5k+8,5k+5,\ldots,5 , 2 ,\ldots, 5k+7, 5k+4 ,\ldots,
4,1,\ldots, 5k+6, 5k+9]$\\

$(0,4,0,5k+3)$ & $[0,\ldots, 5k+5,5k+3 ,\ldots,3,1,\ldots,
5k+6,5k+4 ,\ldots, 4 ,2,\ldots,5k+7]$\\

$(0,4,0,5k+4)$ &
$[0,\ldots,5k+5,5k+7,\ldots,2,4,\ldots,5k+4,5k+6,\ldots,1,3,\ldots,5k+8]$\\

$(0,6,0,5k+1)$&
$[0,2,4,\ldots,5k+4,5k+6,\ldots,1,3,\ldots,5k+3,5k+5,\ldots,5,7,\ldots,5k+7]$\\

 $(0,6,0,5k+5)$ &
$[0,\ldots,5k+10,5k+8,\ldots,3,1,\ldots,5k+6,5k+4,\ldots,4,2,\ldots,5k+7,5k+9,$ \\
&$5k+11]$
\\

$(0,6,1,5k+1)$ &
$[0,2,4,\ldots,5k+4,5k+7,\ldots,7,5,\ldots,5k+5,5k+3,\ldots,3,
  1,\ldots,5k+6,5k+8]$\\

$(0,8,0,5k+2)$ & $[0,2, 4,\ldots,5k+9, 5k+7,\ldots,7,
5,\ldots,5k+5,5k+3,\ldots,3,1,\ldots,5k+6,5k+8,$\\
&$  5k+10] $\\

$(1,0,4,5k+2)$ & $[0, 3 ,\ldots, 5k+3, 5k+6 ,\ldots,1,4 ,\ldots,5k+4 , 5k+5
,\ldots,5 ,2 ,\ldots,5k+7]$\\
\end{tabular}
\end{footnotesize}

\caption{Some perfect linear realizations of $\{1^a,2^b,3^c,5^d\}$.} \label{TablePerfect}
\medskip

\begin{footnotesize}
\begin{tabular}{ll}
$(a,b,c,d)$ & $r\{1^a,2^b,3^c,5^d\}$\\\hline
$(0,1,4,5k+3)$ & $[
0 , 3,\ldots, 5k+8,5k+6,\ldots, 1,4,\ldots, 5k+4,5k+7,\ldots, 2,5,\ldots, 5k+5]$\\
$(0,1,5,5k+5)$ &
$[0,3,\ldots,5k+8,5k+11,\ldots,1,4,\ldots,5k+9,5k+7,5k+10,\ldots,5,$\\
&$2,\ldots,5k+2]$\\

 $(0,2,3,5k+2)$ &
 $[0,2,\ldots,5k+7,5k+4,\ldots,4,1,\ldots,5k+6,5k+3,\ldots,3,5,\ldots,5k+5]$\\

 $(0,4,0,5k+2)$ &
 $[0,\ldots,5k+5,5k+3,\ldots,3,1,\ldots,5k+6,5k+4,\ldots,4, 2,\ldots,5k+2]$\\

$(0,5,0,5k+5)$ &
$[0,\ldots,5k+10,5k+8,\ldots,3,1,\ldots,5k+6,5k+4,\ldots,4,2,\ldots,5k+7,5k+9]$\\
 $(1,0,4,5k+3)$ &
$[0,3,\ldots,5k+8,5k+5,\ldots,5,2,\ldots,5k+7,5k+6,\ldots,1,4,\ldots,5k+4]$\\
\end{tabular}
\end{footnotesize}

\caption{Some linear realizations of $\{1^a,2^b,3^c,5^d\}$.}
\label{TableLinear}
\end{table}
\end{center}

\subsection{Case $a=0$}

As explained in Section 2, it is natural to start investigating the case $a=0$, that
is we consider the multiset
$L=\{2^b,3^c,5^d\}$.
\begin{lem}\label{bc}
We have the following linear realizations of $L=\{2^b,3^c,5^d\}$:
\begin{enumerate}
\item[\rm{(i)}] $r\{2^b,3,5^2\}$ exists if, and only if, $b\geq 3$;
\item[\rm{(ii)}] $r\{2^b,3^2,5\}$ exists if, and only if,  $b\geq 2$;
\item[\rm{(iii)}] $r\{2^3,3^c,5\}$ exists if, and only if,  $c\geq 1$;
\item[\rm{(iv)}] $r\{2^3,3^c,5^2\}$ exists if, and only if,  $c\geq 1$;
\item[\rm{(v)}] $r\{2^4,3^c,5\}$ exists if, and only if,  $c\geq 1$;
\item[\rm{(vi)}] $r\{2^4,3^c,5^2\}$ exists for all $c\geq 0$.
\end{enumerate}
\end{lem}

\begin{proof}
Note that, we the aid of a computer, we can obtain the negative cases as
well as the
positive cases $(b,c,d)\in \{(2,2,1), (3,1,1), (3,1,2), (3,2,1),
(3,3,1), (3,4,2), (4,0,2)\}$.
For the other values of $(b,c,d)$, see \cite{PP}.
\end{proof}

\begin{lem}\label{(0,b,c,1)}
Let $b\geq 3$. The list $L=\{2^b,3^c,5 \}$  has a linear
realization if, and only if, either $(b,c)\in \{(3,1), (4,1), (5,0), (5,1), (6,0),
(6,1), (9,1),(10,1)\}$ or $c\geq 2$.
\end{lem}

\begin{proof}
If $c=0$, then we can use only the integer $5$ to switch from one congruence class modulo
$2$
to the other one, so, by a direct calculation, one can check that $\{2^b,5\}$
is linearly realizable if, and only if, $b=5,6$. Similarly, it is not hard to see that $\{2^b,3,5\}$ has a linear
realization if, and only if, $b=3,4,5,6,9,10$. Hence, we may assume $c\geq 2$.
Furthermore, by Lemma \ref{bc} (iii) and (v) we may also
assume $b\geq 5$.

Suppose $b=5$.
The cases $c=2,3$ are a simple exercise and for $c\geq 4$ we have
$r\{2^5,3^c,5\}=R\{2^2,3^3\}+r\{2^3,3^{(c-3)},5\}$.
Hence the existence follows from Theorem \ref{KDF} (iii) and Lemma \ref{bc} (iii).

Now, suppose $b\geq 6$. By Lemma \ref{bc} (ii) we may assume $c\geq 3$.
For $c=3,4,5$ we have $r\{2^b,3^3,5\}  =
R\{2^4,3\}+r\{2^{(b-4)},3^2,5\}$, $r\{2^b,3^4,5\}  =
R\{2^4,3^2\}+r\{2^{(b-4)},3^2,5\}$ and $r\{2^b,3^5,5\}  =
R\{2^2,3^3\}+r\{2^{(b-2)},3^2,5\}$.
Hence the existence follows from Theorem \ref{KDF} (iii) and Lemma \ref{bc} (ii).
For $c\geq 6$, we have
$r\{2^b,3^c,5\}=R\{2^2,3^3,5\}+r\{2^{(b-2)},3^{(c-3)}\},$
so the existence follows from  Theorem \ref{KDF} (i).
\end{proof}

\begin{lem} \label{(0,b,c,2)}
Let $b \geq 3$. The list $L=\{2^b,3^c,5^2 \}$  has a linear realization if, and
only if, either $c=0$ and $b\in\{4,5,7,8,11,12\}$ or $c\geq 1$.
\end{lem}

\begin{proof}
For $c=0$, by a direct calculation, one can check that $\{2^b,5^2\}$ is linearly realizable if, and only if,
$b\in\{4,5,7,8,11,12\}$.
For $c=1$ the assertion follows from Lemma \ref{bc} (i).
Hence, we may assume $c\geq 2$ and, by Lemma \ref{bc} (iv) and
(vi), we may also assume $b\geq 5$.

Suppose $b=5,\ldots,8$. It is easy to construct a linear realization if $c=2,3$ and for $c\geq 4$, we have
$r\{2^b,3^c,5^2\}=R\{2^2,3^3\}+r\{2^{(b-2)},3^{(c-3)},5^2\}$,
so the result follows from Theorem \ref{KDF} (iii) and Lemma \ref{bc} (iv) and (vi).

Now, assume $b\geq 9$. We have $r\{2^b,3^c,5^2\}=R\{2^6,5\}+r\{2^{(b-6)},3^c,5\}$,
so the existence follows from Lemma \ref{(0,b,c,1)}.
\end{proof}

\begin{lem}\label{(0,b,0,d)}
The list $L=\{2^b,5^d\}$ is linearly realizable in each of the following cases:
\begin{enumerate}
\item[\rm{(i)}] $d=5k+3$ and $b\geq 7$;
\item[\rm{(ii)}] $d=5k+4$ and $b\geq 10$;
\item[\rm{(iii)}] $d=5k+5$ and $b\geq 8$;
\item[\rm{(iv)}] $d=5k+6$ and $b\geq 5$;
\item[\rm{(v)}] $d=5k+7$ and  $b\geq 7$;
\end{enumerate}
for all $k\geq 0$.
\end{lem}

\begin{proof}
First of all we show that $\{2^b,5^3\}$ is linearly realizable if, and only if, $ b\geq
4$: for $b\leq 4$ we obtain the statement with the aid of a computer and for $b\geq 5$ we refer to \cite{PP}.
Now, we start with the case (iv) since this result will be useful in other cases.

(iv) Let $d=5k+6$. For $b=6$ look at Table \ref{TablePerfect} and for $b=5,7$ look
at \cite{PP}.\\
For $b\geq 8$ we have $r\{2^b,5^{(5k+6)}\}=R\{2^4,5^{(5k+3)}\}+r\{2^{(b-4)},5^3\}$,
so the existence follows from Table \ref{TablePerfect}.

(i) Let $d=5k+3$. For $b=7$ we refer to \cite{PP}. For $b=8,9$ we have
$R\{2^8,5^{(5k+8)}\}=R\{2^4,5^4\}+R\{2^4,5^{(5k+4)}\}$ and
$r\{2^9,5^{(5k+8)}\}=R\{2^4,5^3\}+r\{2^5,5^{(5k+5)}\}$.\\
For $b\geq10$, $r\{2^b,5^{(5k+8)}\}=R\{2^6,5^{(5k+5)}\}+r\{2^{(b-6)},5^3\}$ and
so the existence follows from Tables \ref{TablePerfect} and \ref{TableLinear}.

(ii) Let $d=5k+4$ and $b\geq10$. Since $r\{2^b,5^{(5k+4)}\}=R\{2^6,5^{(5k+1)}\}+r\{2^{(b-6)},5^3\}$,
the existence follows from Table \ref{TablePerfect}.

(iii) Let $d=5k+5$. For $b=8,\ldots,11$ we have
 $r\{2^8,5^{(5k+5)}\}=R\{2^4,5^3\}+r\{2^4,5^{(5k+2)}\}$,
 $r\{2^9,5^{(5k+5)}\}=R\{2^4,5^4\}+r\{2^5,5^{(5k+1)}\}$,
 $R\{2^{10},5^{(5k+5)}\}=R\{2^4,$ $5^4\}+R\{2^6,5^{(5k+1)}\}$ and
$r\{2^{11},5^{(5k+10)}\}=R\{2^4,5^4\}+r\{2^7,5^{(5k+6)}\}$.
Also, we have directly checked the existence of a linear realization of $\{2^{11},5^5\}$.\\
Finally, for $b\geq12$ we have
$r\{2^b,5^{(5k+5)}\}=R\{2^8,5^{(5k+2)}\}+r\{2^{(b-8)},5^3\}$.
So the existence follows from Tables \ref{TablePerfect} and \ref{TableLinear} and case (iv).

(v) Let $d=5k+7$. A linear realization for $b=7$ is in \cite{PP}.
For $b\geq 8$, since $r\{2^b,5^{(5k+7)}\}=R\{2^4,5^{(5k+4)}\}+r\{2^{(b-4)},5^3\}$,
the result follows from Table \ref{TablePerfect}.
\end{proof}

\begin{prop}\label{(0,1,c,d)}
Assume $d\geq 7$ and $L\neq \{2,3^4,5^{(5k+10)}\}$. The list
$L=\{2,3^c,5^d\}$ is linearly realizable if, and only if, $c\geq 4$.
\end{prop}

\begin{proof}
When $c\leq 2$ the list $L$ cannot be linearly realizable by Remark \ref{a+b+c}.
So, assume $c\geq3$. It is a routine to check that the list $\{2,3^3,5^d\}$ has no linear
realization for every $d\geq 6$.
Also, it is not hard to construct linear realizations of
the lists $\{2,3^3,5^4\}, \{2,3^3,5^5\}$ and $\{2,3^5,5^4\}$. Furthermore,
we obtained in \cite{PP} a linear realization of $L$ when $(c,d)\in
\{(3t+6,4),(3t+8,4),(3t+4,5),(3t+6,5)\}$ for all $t\geq0$.
So, when $c\equiv 0,2\pmod 3$, the list $\{2,3^c,5^4\}$ is linearly
realizable if, and only if,  $c\geq 3$ and when $c\equiv 0,1\pmod 3$, the list
$\{2,3^c,5^5\}$ if, and only if, $c\geq 3$.\\
We recall that by hypothesis $d\geq7,$ anyway sometimes our constructions
hold also for $d\leq6$.
 We split the proof into five cases according to the congruence class of $d$
modulo 5.

i) Let $d\equiv2 \pmod 5$. For $c=4,\ldots,7$ we refer to \cite{PP} and for $c=8$ we have
$r\{2,3^8,5^{(5k+7)}\}=R\{3^4,5^4\}+r\{2,3^4,5^{(5k+3)}\}$, hence the existence follows
from Table \ref{TableLinear}.\\
Assume $c\geq 9$. If $c\equiv 1 \pmod 3$, then
$r\{2,3^c,5^{(5k+7)}\}=R\{3^6,5^{(5k+2)}\}+r\{2,3^{(c-6)},5^5\}$. If $c\equiv 0,2\pmod
3$, $r\{2,3^c,5^{(5k+7)}\}=R\{3^6,5^{(5k+3)}\}+r\{2,3^{(c-6)},$ $5^4\}$.
Now, the existence follows from Table \ref{TablePerfect}.

ii) Let $d\equiv4 \pmod5$. For $c=4,\ldots,8$ look at \cite{PP} and for $c=9$ we have
$r\{ 2,3^9,5^{(5k+9)}\}=R\{3^4,5^4\}+r\{2,3^5,5^{(5k+5)}\}$.\\
 For $c\geq 10$,
$r\{2,3^c,5^{(5k+9)}\}=R\{3^6,5^2\}+r\{2,3^{(c-6)},5^{(5k+7)}\}$.
So the existence follows from Tables \ref{TablePerfect}, \ref{TableLinear} and case i).

iii) Let $d\equiv1 \pmod5$. For $c=4,\ldots,7$, we constructed a linear realization in
\cite{PP} and for $c=8$ we have
$r\{2,3^8,5^{(5k+6)}\}=R\{3^4,5^4\}+r\{2,3^4,5^{(5k+2)}\}$.\\
Furthermore, for all $t\geq 0$ we have $
r\{2,3^{(3t+9)},5^{(5k+6)}\}=R\{3^6,5^{(5k+2)}\}+r\{2,3^{(3t+3)},5^4\}$,
$r\{2,3^{(3t+10)},5^{(5k+11)}\}=R\{3^6,5^2\}+r\{2,3^{(3t+4)},5^{(5k+9)}\}$ and
$r\{2,3^{(3t+11)},5^{(5k+6)}\}$ $=R\{3^6,5^{(5k+2)}\}+r\{2,3^{(3t+5)},5^4\}$.
Hence the existence follows from Table \ref{TablePerfect} and cases i), ii).

iv) Let $d\equiv3\pmod5$. For $c=4$ look at Table \ref{TableLinear} and for $c=5,6,7$
look at \cite{PP}. For $c=8$ we have
$r\{2,3^8,5^{(5k+8)}\}=R\{3^4,5^4\}+r\{2,3^4,5^{(5k+4)}\}$.\\
Furthermore, for all $t\geq 0$ we have
$r\{2,3^{(3t+9)},5^{(5k+8)}\}=R\{3^6,5^{(5k+3)}\}+r\{2,$ $3^{(3t+3)},5^5\}$
$r\{2,3^{(3t+10)},5^{(5k+8)}\}=R\{3^6,5^{(5k+3)}\}+r\{2,3^{(3t+4)},5^5\}$ and
$r\{2,3^{(3t+11)},5^{(5k+8)}\}$ $=R\{3^6,5^{(5k+2)}\}+r\{2,3^{(3t+5)},5^6\}$.
Hence the existence follows from Table \ref{TablePerfect} and cases ii), iii).

v) Let $d\equiv0\pmod5$. For $c=5$ look at Table \ref{TableLinear} and for $c=6,7$ we refer to \cite{PP}. For $c=8,9$ we have
$r\{2,3^c,5^{(5k+5)}\}=R\{3^4,5^4\}+r\{2,3^{(c-4)},5^{(5k+1)}\}$.\\
For $c\geq 10$ we have $r\{2,3^c,5^{(5k+10)}\}=R\{3^6,5^3\}+r\{2,3^{(c-6)},5^{(5k+7)}\}$.
So the existence follows from Table \ref{TablePerfect} and cases i), iii).
\end{proof}

\begin{prop}\label{(0,2,c,d)}
Assume $d\geq 3$ and $L\neq \{2^2,3^3,5^{(5k+8)}\}$. The list $L=\{2^2,3^c,5^d\}$ is
linearly realizable if, and only if, either $c=2$ and $d\neq 5k+7,5k+8,5k+9$ or $c\geq
3$.
\end{prop}

\begin{proof}
Observe that $L$ is not linearly realizable when $c=0,1$ (see Remark \ref{a+b+c}).
First, if either $d=3$ and $c\geq 4$
or $d=4$ and $c\geq 3$ we described a linear realization of $L$ in \cite{PP}.
A linear
realization of
$\{2^2,3^2,5^3\}, \{2^2,3^3,5^3\}$ and $\{2^2,3^ 2, 5^4\}$ can be easily constructed. It follows that
the lists $\{2^2,3^c,5^3\},\{2^2,3^c,5^4\}$ are linearly realizable if, and only if, $
c\geq 2$.\\
Now, we split the proof into five cases according to the congruence class of
$d$ modulo $5$.

i) Let $d\equiv1\pmod5$. For $c=2,\ldots,6$ we refer to \cite{PP} and for $c=7$ we have
$r\{2^2,3^7,5^{(5k+6)}\}=R\{3^4,5^4\}+r\{2^2,3^3,5^{(5k+2)}\}$.\\
Furthermore, for $c\geq 8$ we have
$r\{2^2,3^c,5^{(5k+6)}\}=R\{3^6,5^{(5k+3)}\}+r\{2^2,3^{(c-6)},5^3\}$. Hence the existence
follows from Tables \ref{TablePerfect} and \ref{TableLinear}.

ii) Let $d\equiv0\pmod5$. For $c=2,\ldots,5$ look at \cite{PP}. For $c=6,7$ we have
$r\{2^2,3^6,5^{(5k+5)}\}=R\{3^4,5^4\}+r\{2^2,3^2,5^{(5k+1)}\}$ and
$R\{2^2,3^7,5^{(5k+5)}\}=R\{3^4,5^4\}$ $+R\{2^2,3^3,5^{(5k+1)}\}$.\\
For $c\geq 8$, we have
$r\{2^2,3^c,5^{(5k+5)}\}=R\{3^6,5^{(5k+2)}\}+r\{2^2,3^{(c-6)},$ $5^3\}$.
So the existence follows from Table \ref{TablePerfect} and case i).

iii) Let $d\equiv2\pmod5$. It is a routine to check that $r\{2^2,3^2,5^{(5k+7)}\}$ does not exist. A
realization of $\{2^2,3^3,5^{(5k+2)}\}$ is described in Table \ref{TableLinear}. For
$c=4,\ldots,7$ look at \cite{PP}.\\
For $c\geq 8$, we have
$r\{2^2,3^c,5^{(5k+7)}\}=R\{3^6,5^2\}+r\{2^2,3^{(c-6)},5^{(5k+5)}\}$.
So the existence follows from Table \ref{TablePerfect} and case ii).

iv) Let $d\equiv4\pmod5$.  It is a routine to check that $r\{2^2,3^2,5^{(5k+9)}\}$ does not exist.
For $c=3,4,5$ we refer to \cite{PP} and for $c\geq 6$ we have
$\{2^2,3^c,5^{(5k+9)}\}=R\{3^4,5^4\}+r\{2^2,3^{(c-4)},5^{(5k+5)}\}$,
so the existence follows from case ii).

v) Let $d\equiv 3\pmod5$.  It is a routine to check that $r\{2^2,3^2,5^{(5k+8)}\}$ does
not exist.
For
$c=4,5,6$ look at \cite{PP} and for $c=7$ we have
$r\{2^2,3^7,5^{(5k+8)}\}=R\{3^4,5^4\}+r\{2^2,3^3,5^{(5k+4)}\}$.\\
Finally, for $c\geq 8$ we have
$r\{2^2,3^c,5^{(5k+8)}\}=R\{3^6,5^3\}+r\{2^2, $ $3^{(c-6)},$ $5^{(5k+5)}\}$,
hence the existence follows from Table \ref{TablePerfect} and cases ii), iv).
\end{proof}

We can now prove Theorem \ref{lin a=0}.
\begin{proof}[Proof of Theorem \emph{\ref{lin a=0}}]
In view of Lemmas \ref{(0,b,c,1)} and \ref{(0,b,c,2)} we may assume $d\geq 3$.
First, suppose $b=3$.
For $c=1$, the existence of a $r\{2^3,3,5^3\}$ can be easily obtained; on the other hand it is not hard
to see that $L=\{2^3,3,5^{(5k+8)}\}$ does not have a linear
realization. For $d=5k,5k+1,5k+2,5k+4$ we construct the linear realizations in \cite{PP}.
Also, the cases $c=2,3,4$ are dealt with in \cite{PP}.

Now, assume $c=5$.
If $d=5k+2$ we refer to \cite{PP}. If $d=3$ we obtain the result by direct constructions.
For $d\geq 4$, with $d\not\equiv 2\pmod 5$ we have
$r\{2^3,3^5,5^d\}=R\{3^4,5^4\}+r\{2^3,3,5^{(d-4)}\}$.\\
Now, suppose $c\geq 6$.  For $d=3,4,5$, if $c=6$ it is easy to construct a linear realization, if $c\geq7$ we have
$r\{2^3,3^c,5^3\}=R\{3^6,5^2\}+r\{2^3,3^{(c-6)},5\}$,
$r\{2^3,3^c,5^4\}=R\{3^6,5^2\}+r\{2^3,3^{(c-6)},$ $5^2\}$ and
$r\{2^3,3^c,5^5\}=R\{3^6,5^3\}+r\{2^3,3^{(c-6)},5^2\}$.
In this case the existence follows from Table \ref{TablePerfect} and
Lemma \ref{bc}  (iii), (iv).
For $d\geq 6$ we have $r\{2^3,3^c,5^d\}=R\{3^4,5^4\}+r\{2^3,3^{(c-4)},5^{(d-4)}\}$.

Next, assume $b\geq 4$. We split the proof into five cases according to the congruence
class of $d$ modulo $5$.

i) If $d\equiv 3\pmod 5$ we start considering the case $b=4$. For $c=1,\ldots,5$ we
refer to \cite{PP}, for $c=6$ we have
$r\{2^4,3^6,5^{(5k+3)}\}  =R\{2^2,3^3,5\}+r\{2^2,3^3,5^{(5k+2)}\}$ and for $c\geq 7$
we obtain $R\{2^4,3^c,5^{(5k+3)}\}  = R\{2^2,3^3\}+R\{2^2,3^{(c-3)},5^{(5k+3)}\}$.
The existence follows from Table \ref{TableLinear}, Theorem
\ref{KDF} (iii) and Proposition \ref{(0,2,c,d)}.\\
Consider the case $b=5$: for $c=1,\ldots,4$ look at \cite{PP} and for $c\geq 5$ we
have $r\{2^5,3^c,5^{(5k+3)}\}$ $ =
R\{2^2,3^3\}+r\{2^3,3^{(c-3)},5^{(5k+3)}\}.$
So the existence follows from Theorem \ref{KDF} (iii) and above construction for $b=3$.\\
Now, let $b=6$. For $c=1$ we refer to \cite{PP}, for $c=2$ we have
$r\{2^6,3^2,5^{(5k+3)}\}  = R\{2^4,5^{(5k+3)} \}+r\{ 2^2,3^2 \}$ and for $c\geq 3$ we
have $r\{2^6,3^c,5^{(5k+3)}\}  =  R\{2^2,3^3\}+r\{2^4,3^{(c-3)},$ $5^{(5k+3)} \}$.
Thus, the existence follows from Table \ref{TablePerfect} and Theorem \ref{KDF} (iii).\\
If $b=7$ we have $r\{2^7,3^c,5^{(5k+3)}\} = R\{2^4,5^{(5k+3)}\}+r\{2^3,3^c\}$
for $c=1,\ldots,8$ and $r\{2^7,3^c,5^{(5k+3)} \}  = R\{3^6,5^{(5k+3)}\}+r\{2^7,3^{(c-6)}\}$
for $c \geq 9$.
So, the result follows from Table \ref{TablePerfect} and Theorem \ref{KDF} (i), (ii).\\
Assume $b\geq 8$. If $c=1,2$ we have
$$\left.\begin{array}{l}
R\{2^8,3,5^{(5k+3)}\}  = R\{2^4,5^{(5k+3)}\}+R\{2^4,3\},\\
r\{2^b,3,5^{(5k+3)}\}  =
R\{2^6,5^{(5k+1)}\}+r\{2^{(b-6)},3,5^2\}\quad \textrm{for }  b\geq 9,\\
r\{2^8,3^2,5^{(5k+3)}\}= R\{2^4,3\}+r\{2^4,3,5^{(5k+3)}\},\\
r\{2^b,3^2,5^{(5k+3)}\} =
R\{2^6,3,5^{(5k+1)}\}+r\{2^{(b-6)},3,5^2\}\quad \textrm{for }  b\geq 9,
\end{array}\right.$$
hence the existence follows from Table \ref{TablePerfect}, Theorem \ref{KDF} (iii) and
Lemma \ref{bc} (i).
If $c\geq 3$, then
$r\{2^b,3^c,5^{(5k+3)}\}=R\{2^4,5^{(5k+3)}\}+r\{2^{(b-4)},3^c\}$,
so the result follows from Table \ref{TablePerfect} and Theorem \ref{KDF} (i).

ii) If $d\equiv 2\pmod 5$ we can suppose $d\geq 7$, anyway in some cases we construct
linear realizations  also for $d=2$.\\
Let $c=1$. For $b=4,\ldots,7$ look at \cite{PP} and for $b\geq 8$ we have
$r\{2^b,3,5^{(5k+7)}\}=R\{2^4,5^4\}+r\{2^{(b-4)},3,5^{(5k+3)}\}$.
So the existence follows from Table \ref{TablePerfect} and case i).\\
 Take now $c=2$. For
$b=4,5,6$ we construct the linear realizations in \cite{PP} and for $b\geq 7$ we have
$r\{2^b,3^2,5^{(5k+7)}\}=R\{2^4,5^4\}+r\{2^{(b-4)},3^2,5^{(5k+3)}\}$.
Hence the result follows from Table \ref{TablePerfect} and
case i).\\
Take $c=3$. For $b=4,5$ look at \cite{PP}. For $b=6$ we have
$r\{2^6,3^3,5^{(5k+2)}\} =  R\{2^2,3^3\}+r\{2^4,5^{(5k+2)}\}$ and for $b\geq 7$ we have
$r\{2^b,3^3,5^{(5k+7)}\}=R\{2^4,5^4\}+r\{2^{(b-4)},3^3,5^{(5k+3)}\}$.
Hence the existence follows from Theorem \ref{KDF} (iii), Tables  \ref{TablePerfect}
and \ref{TableLinear} and
case i).\\
Now take $c=4$. In this case we have
$$\left.\begin{array}{l}
R\{2^4,3^4,5^{(5k+7)}\}  =  R\{3^4,5^4\}+R\{2^4,5^{(5k+3)}\},\\
r\{2^b,3^4,5^{(5k+7)}\}
=  R\{2^2,3^3\}+r\{2^{(b-2)},3,5^{(5k+7)}\}\quad \textrm{if } b\geq 5,
\end{array}\right.$$
so the existence follows from Theorem \ref{KDF} (iii) and Table \ref{TablePerfect}.\\
Finally, for $c\geq 5$ we have
$r\{2^b,3^c,5^{(5k+7)} \}=R\{3^4,5^4\}+r\{2^b,3^{(c-4)},5^{(5k+3)}\}$,
so the existence follows from case i).

iii) Let $d\equiv1\pmod 5$. We can suppose $d\geq 6$, anyway
 in some cases we construct linear realizations also for $d=1$.\\
Let $b=4$. For $c=1,2,3$ we refer to \cite{PP} and for $c\geq 4$ we have
$r\{2^4,3^c,5^{(5k+6)}\}  =
R\{3^4,5^4\}+r\{2^4,3^{(c-4)},5^{(5k+2)}\}$.
So the result follows from case ii) and Table \ref{TableLinear}.\\
Let $b=5$. For $c=1,2,3$ look at \cite{PP} and for $c\geq 4$ we have
$r\{2^5,3^c,5^{(5k+1)}\}  = R\{2^2,3^3\}+r\{2^3,3^{(c-3)},5^{(5k+1)}\}$,
so the existence
follows from Theorem \ref{KDF} (iii)
and the above construction for $b=3$.\\
For $b=6$, if $c=1$ look at Table \ref{TablePerfect}, for $c=2$ look at \cite{PP} and
for $c\geq 3$ we have
$r\{2^6,3^c,5^{(5k+6)}\}=R\{2^4,3\}+r\{2^2,3^{(c-1)},5^{(5k+6)}\}$.
Hence the existence follows from Theorem \ref{KDF} (iii) and Proposition \ref{(0,2,c,d)}.\\
For $b\geq 7$, we have
$r\{2^b,3^c,5^{(5k+6)}\}=R\{2^4,5^{(5k+4)}\}+r\{2^{(b-4)},3^c,5^2\}$.
In this case the existence follows from Table \ref{TablePerfect} and Lemma \ref{(0,b,c,2)}.

iv) Let $d\equiv 0\pmod 5$.  We can assume $d\geq 5$, anyway sometimes we
construct linear realizations also for $d=0$.\\
Consider $b=4$. For $c=1,\ldots,4$ we list the linear realizations in \cite{PP} and for
$c\geq 5$ we have
$r\{2^4,3^c,5^{(5k+5)}\}=R\{2^2,3^3\}+r\{2^2,3^{(c-3)},5^{(5k+5)}\}$, hence the existence
follows from Theorem \ref{KDF} (iii) and Proposition \ref{(0,2,c,d)}.\\
Let $b=5$. For $c=1,2,3$ we refer to \cite{PP} and for $c\geq 4$ we have
$r\{2^5,3^c,5^{(5k+5)}\}=R\{2^2,3^3\}+r\{2^3,3^{(c-3)},5^{(5k+5)}\}$,
so the existence follows from Theorem \ref{KDF} (iii) and
the above construction for $b=3$.\\
Let $b=6$. If $c=1$ look at \cite{PP}. If $c=2,3$ we have
$r\{2^6,3^2,5^{(5k+5)}\}  =
R\{2^4,5^{(5k+4)}\}+r\{2^2,3^2,5\}$ and
$r\{2^6,3^3,5^{(5k+5)}\}  = R\{2^4,3,5^{(5k+5)}\}+r\{2^2,3^2\}$,
so the result follows from Table \ref{TablePerfect}, Theorem \ref{KDF} (iii)
and Lemma \ref{bc} (ii).
For $c\geq 4$, we have
$r\{2^6,3^c,5^{(5k+5)}\}=R\{2^2,3^3\}+r\{2^4,3^{(c-3)},5^{(5k+5)}\}$.
Here the existence follows from Table \ref{TablePerfect}.\\
For $b\geq 7$ we have
$r\{2^b,3^c,5^{(5k+5)}\}=R\{2^4,5^{(5k+3)}\}+r\{2^{(b-4)},3^c,5^2\}$,
so the existence follows from Table \ref{TablePerfect}
and Lemma \ref{(0,b,c,2)}.

v) Let $d\equiv 4\pmod 5$. Consider $b=4$. For $c=1,2,3$ we refer to \cite{PP}. For $c=4$
we have $r\{2^4,3^4,5^{(5k+4)}\}  = R\{2^2,3^2,5^3\}+r\{2^2,3^2,5^{(5k+1)}\}$ and for
$c\geq 5$ we have
$r\{2^4,3^c,5^4\}  =R\{3^4,5^4\}+r\{2^4,3^{(c-4)}\}$ and $r\{2^4,3^c,5^{(5k+9)}\} =
R\{3^4,5^4\}+r\{2^4,3^{(c-4)},5^{(5k+5)}\}$.
Hence the existence follows from Theorem
\ref{KDF} (i), (iii), Proposition \ref{(0,2,c,d)} and the above case iv).\\
Suppose now $b\geq 5$.
Let $c=1$. If $b=5,6$ we refer to \cite{PP}. For $b=7,\ldots,10$ we have
$r\{2^b,3,5^{(5k+4)}\}=R\{2^4,5^{(5k+3)}\}+r\{2^{(b-4)},3,5\}$ and for $b\geq 11$ we have
$r\{2^b,3,5^{(5k+4)}\}=R\{2^8,5^{(5k+2)}\}+r\{2^{(b-8)},3,5^2\}$.
Hence, the existence follows from Table \ref{TablePerfect}, Lemma \ref{bc} (i) and Lemma
\ref{(0,b,c,1)}. Assume $c\geq 2$. If $(b,c)=(5,2)$ we construct the linear realization
in \cite{PP}. Furthermore, we have
$$\left.\begin{array}{l}
r\{2^5,3^c,5^{(5k+4)}\}= R\{2^2,3^2,5^3\}+r\{2^3,3^{(c-2)},5^{(5k+1)}\},\quad \textrm{if }
c\geq 3,\\
r\{2^6,3^c,5^4\}=R\{2^4,5^4\}+r\{2^2,3^c\}, \quad\textrm{if  } c=2,3,\\
r\{2^6,3^c,5^4\}=R\{2^2,3^3,5\}+r\{2^4,3^{(c-3)},5^3\}\quad \textrm{if } c\geq 4,\\
r\{2^6,3^c,5^{(5k+9)}\}=R\{2^4,5^3\}+r\{2^2,3^c,5^{(5k+6)}\},\\
r\{2^b,3^c,5^{(5k+4)}\}=R\{2^4,5^{(5k+3)}\}+r\{2^{(b-4)},3^c,5\},\quad \textrm{if }
b\geq 7.
\end{array}\right.$$
The existence follows from Table \ref{TablePerfect}, Theorem \ref{KDF} (iii), Lemma
\ref{(0,b,c,1)},
Proposition \ref{(0,2,c,d)}, case i) and the above construction for $b=3$.
\end{proof}

\subsection{Case $a=1$}

In the following we consider only lists of the form $L=\{1,2^b,3^c,5^d\}$.

\begin{prop}\label{(1,0,c,d)}
Let $d\geq 1$. The list $L=\{1,3^c,5^d \}$ has a
linear realization if, and only if, $(c,d)$ satisfies one of the following:
\begin{enumerate}
\item[\rm{(i)}] $c=3$ and $d\neq 5k+1,5k+8,5k+10$, for all $k\geq 0$.
\item[\rm{(ii)}] $4\leq c\leq 6$;
\item[\rm{(iii)}] $c=7$ and $d\neq 2$;
\item[\rm{(iv)}] $c\geq 8$ and $d\geq 3$.
\end{enumerate}
\end{prop}

\begin{proof}
If $c\leq 2$, the result follows from Remark \ref{a+b+c}.
If $d=1$, we obtain that $\{1,3^c,5\}$ is linearly realizable if, and only if,
$c=4,5,6,7$. Futhermore, $\{1,3^c,5^2\}$ is linearly realizable if, and only if,
$c=3,4,5,6$.

Now, consider the cases $d=3,4$. It is not hard to see that the lists
$\{1,3^3,5^3\}, \{1,3^3,5^4\},$
$\{1,3^4,5^3\}$ and $\{1,3^4,5^4\}$ admit a linear realization.
Furthermore, we obtain a linear realization if $d=3,4$ and $c\geq 5$ in \cite{PP}. So in the following
we can restrict our investigation at $c\geq 3$ and $d\geq5$.
Anyway in some cases
our constructions hold also for $d=1,\ldots,4$.
We split the proof into five cases according to the congruence class of $d$ modulo
$5$.

1) Let $d\equiv1\pmod5$. It is not hard to see that a linear realization of
$\{1,3^3,$ $5^{(5k+1)}\}$ does not exist.
For $c=4,\ldots,7$ we refer to \cite{PP}. For $c=8$, we have
$R\{1,3^8,5^{(5k+6)}\}=R\{3^4,5^4\}+R\{1,3^4,5^{(5k+2)}\}$
and for $ c\geq 9$,
$r\{1,3^c,5^{(5k+6)}\}=R\{3^6,5^{(5k+2)}\}+r\{1,3^{(c-6)},$ $5^4\}$.
So the existence follows from Table \ref{TablePerfect} and the above construction.

2) Let $d\equiv 2\pmod 5$. For $c=4$ a linear realization is contained in Table
\ref{TablePerfect} and for $c=3,5,6,7$ look at \cite{PP}.
For $c=8$ we have
$r\{1,3^8,5^{(5k+7)}\}=R\{3^4,5^4\}+r\{1,3^4,5^{(5k+3)}\}$
and for $c\geq 9$,
$r\{1,3^c,5^{(5k+7)}\}=R\{3^6,5^{(5k+3)}\}+r\{1,3^{(c-6)},5^4\}$.
Hence the existence follows from Tables \ref{TablePerfect} and \ref{TableLinear} and the
above construction.

3) Let $d\equiv 0\pmod 5$, $d>0$. One can check that a linear realization of
$\{1,3^3,5^{(5k+10)}\}$ cannot exist, while a $r\{1,3^3,5^{5}\}$ can be easily constructed.
For $c=4,\ldots,7$  look at \cite{PP}, for $c=8$ we have
$r\{1,3^8,5^{(5k+5)}\}=R\{3^4,5^4\}+r\{1,3^4,5^{(5k+1)}\}$
and for $c\geq 9$, $r\{
1,3^c,5^{(5k+5)}\}=R\{3^6,5^{(5k+2)}\}+r\{1,3^{(c-6)},5^3\}$. In this case the
existence follows from Table \ref{TablePerfect}, case 1) and the above
construction.

4) Let $d\equiv 4 \pmod 5$. For $c=3,\ldots,7$ we refer to \cite{PP}
and for $c\geq 8$ we have
$r\{1,3^c,5^{(5k+9)}\}=R\{3^4,5^4\}+r\{1,3^{(c-4)},5^{(5k+5)}\}$.
So the existence follows from case 3).

5) Let $d\equiv 3\pmod 5$. It can be seen that a linear realization of
$\{1,3^3,5^{(5k+8)}\}$
does not exist. For $c=4$ the realization is contained in Table \ref{TableLinear},
for $c=5,6$ they are contained in \cite{PP}
and for $c\geq 7$,
$r\{1,3^c,5^{(5k+8)}\} =R\{3^4,5^4\} + r\{1,3^{(c-4)},5^{(5k+4)}\}$.
Hence the existence follows from case 4).
\end{proof}

\begin{prop}\label{(1,1,c,d)}
Let  $d\geq 1$. The list $L=\{1,2,3^c,5^d \}$ has a
linear realization if, and only if, either $c=2$ and $d\neq 5k+7,5k+9$,  for
all $k\geq 0$, or $c\geq 3$.
\end{prop}

\begin{proof}
If $c\leq 1$ the assertion follows from Remark \ref{a+b+c}. Now, consider
the cases $d=1,2,3,4$.
It is not hard to solve the cases
$(c,d)\in \{(2,1), (2,2), (3,2), (4,2), (2,3), (3,3),$ $(4,3), (2,4), (3,4),(6,4)\}$.
Furthermore,
we obtained in \cite{PP} a linear realization if either $d=1$ and $c\geq 3$ or $d=2,3$ and
$c=3t+5$ for all $t\geq 0$. Also, we have
$$\left.\begin{array}{l}
r\{1,2,3^{(3t+6)},5^2\}  = R\{3^6,5^2\}+r\{1,2,3^{3t}\},\\
r\{1,2,3^{(3t+7)},5^2\}  = R\{3^6,5^2\}+r\{1,2,3^{(3t+1)}\},\\
r\{1,2,3^{(3t+6)},5^3\}  = R\{3^6,5^3\}+r\{1,2,3^{3t}\},\\
r\{1,2,3^{(3t+7)},5^3\}   =  R\{3^6,5^3\}+r\{1,2,3^{(3t+1)}\},\\
r\{1,2,3^{(3t+4)},5^4\}  =  R\{3^4,5^4\}+r\{1,2,3^{3t}\},\\
r\{1,2,3^{(3t+5)},5^4\}  = R\{3^4,5^4\}+r\{1,2,3^{(3t+1)}\}, \\
r\{1,2,3^{(3t+9)},5^4\}  = R\{3^6,5^2\}+r\{1,2,3^{(3t+3)},5^2\}.
\end{array}\right.$$
Hence the existence follows from Table \ref{TablePerfect} and Theorem \ref{KDF} (vi).
Furthermore, it is not hard to construct a linear realization of the lists $\{1,2,3^c,
5^d\}$ with $d=5,6$ and $2\leq c\leq 5$. For $c\geq 6$
we have
$r\{1,2,3^c,5^5\}=R\{3^4,5^4\}+r\{1,2,3^{(c-4)},5\}$ and
$r\{1,2,3^c,5^6\}=R\{3^4,5^4\}+r\{1,2,3^{(c-4)},5^2\}$. So in the following we can
restrict our investigation at $d\geq 7$.

For $c=2$ and $d=5k, 5k+1, 5k+3$ see \cite{PP}, however
one can check that a $r\{1,2,3^2,5^{(5k+7)}\}$ and a $r\{1,2,3^2,5^{(5k+9)}\}$
does not exist. For $c=3$ we refer to \cite{PP}. For $c\geq 4$ the result follows from
Remark \ref{RemA>a} and Proposition
\ref{(0,1,c,d)}, except when $L=
\{1,2,3^4,5^{(5k+10)}\}$. However, in this case, we have a linear realization, see
\cite{PP}.
\end{proof}

\begin{prop}\label{(1,2,c,d)}
Let $d\geq 1$. The list $L=\{1,2^2,3^c,5^d \}$ has a
linear realization if, and only if, either $c=1$ and $d\neq 5k+8$, for all
$k\geq 0$, or $c\geq 2$.
\end{prop}

\begin{proof}
 If $c=0$ the statement follows from Remark \ref{a+b+c}.
Let $d=1$. It is easy to obtain a linear realization of $\{1,2^2,3,5\}$ and for $c\geq 2$ we refer to \cite{PP}.
Let $d=2$. It is easy to obtain a linear realization of $\{1,2^2,3,5^2\}$
and $\{1,2^2,3^2,5^2\}$, and
for $c\geq 3$ see \cite{PP}.\\
Now, assume $d\geq 3$. For $c=1$ it can be checked that a
$r\{1,2^2,3,5^{(5k+8)}\}$ does not exist, whereas the list $\{1,2^2,3,5^3\}$ can be
linearly  realized.  Also, when $d=5k,5k+1,5k+2,5k+4$ we have a linear realization, see
\cite{PP}.
We also refer to \cite{PP} for the linear realizations of $L$ when either $c=2$ and
$d=5k+2,5k+3,5k+4$ or $(c,d)=(3,5k+3)$.
For the remaining cases, apply Remark \ref{RemA>a} and  Lemma \ref{(0,2,c,d)}.
\end{proof}

\begin{prop}\label{(1,b,0,d)}
Let $d\geq 1$. The list $L=\{1,2^b,5^d \}$ has a linear
realization if, and only if, either $b=3$ and $d\neq 5k+7, 5k+9$
for all
$k\geq 0$, or $b \geq 4$.
\end{prop}

\begin{proof}
For $b=0,1,2$ the non existence follows from Remark \ref{a+b+c}.
It is not hard to construct a linear realization of the lists
$\{1,2^3,5\}, \{1,2^4,5\}$ and $\{1,2^5,5\}$;
also, for $d=1$ and $b\geq 6$, see
\cite{PP}.
Furthermore, we have directly constructed a linear realization of
$\{1,2^b,5^2\}$ for $3\leq b\leq 8$.
For $b\geq 9$ we have $r\{1,2^b,5^2\}=R\{2^6,5\}+r\{1,2^{(b-6)},5\}$.
So, we may assume
$d\geq 3$, anyway in some cases our constructions include
also $d=1,2$.\\
It is a routine to check that a $r\{1,2^3,5^{(5k+7)}\}$ and a $r\{1,2^3,5^{(5k+9)}\}$
cannot exist. On the other hand
a $r\{1,2^3,5^{4}\}$ can be easily obtained.
Now, in view of Remark \ref{RemA>a} and Lemma \ref{(0,b,0,d)} it is sufficient to construct
a linear realization of $L$ in the following cases:
$d=5k+3$ and $3\leq b\leq6$;
$d=5k+4$ and $4\leq b\leq9$;
$d=5k+5$ and $3\leq b\leq7$;
$d=5k+6$ and $b=3,4$;
$d=5k+7$ and $4\leq b\leq6$.
If either $b=3$  and $d=5k,5k+1,5k+3$ or
$b=4$ and $d= 5k,5k+1,5k+2$ or $b=5,6$ and $d=5k,5k+2$ look at \cite{PP}.
Furthermore,
$$\left.\begin{array}{l}
r\{1,2^b,5^{(5k+3)}\}  = R\{2^4,5^{(5k+3)}\} +
r\{1,2^{(b-4)}\}\quad \textrm{if } b= 4,5,6,\\
r\{1,2^b,5^{(5k+4)}\}  = R\{2^4,5^{(5k+4)}\} +
r\{1,2^{(b-4)}\}\quad \textrm{if } b =4,\ldots, 9,\\
r\{1,2^7,5^{(5k+5)}\} = R\{2^4,5^{(5k+4)}\} + r\{1,2^3,5\}.\\
\end{array}\right.$$
The existence follows from Table \ref{TablePerfect} and Theorem \ref{KDF} (v).
\end{proof}

\subsection{Case $a\geq2$}
\begin{lem}\label{a=2}
The multiset
$L=\{1^2,2^b,3^c,5^d \}$ has a linear realization if $(b,c,d)$ is any of the following:
\begin{enumerate}
\item[\rm{(i)}] $(0,2,d)$, $(1,1,d)$, $(2,0,d)$ for all $d\geq 0$;
\item[\rm{(ii)}] $(0,c,1)$, $(0,c,2)$ if and only if $c\geq 2$;
\item[\rm{(iii)}] $(0,3,5k)$, $(0,3,5k+1)$, $(0,3,5k+3)$,
$(1,2,5k+2)$, $(1,2,5k+4)$,
$(2,1,5k+3)$, $(3,0,5k+2)$,
$(3,0,5k+4)$, for all $k\geq 0$.
\end{enumerate}
\end{lem}

\begin{proof}
(i), (iii) See \cite{PP}.
(ii) Let $d=1$. It is easy to obtain a $r\{1^2,3^2,5\}$; for $c\geq 3$ see
\cite{PP}. Suppose now $d=2$. If $c=2,\ldots,5$, to obtain the result is a simple exercise.
For $c\geq 6$ it results
$r\{1^2,3^c,5^2\}=R\{3^6,5^2\}+r\{1^2,3^{(c-6)}\}$
so the existence follows from Table \ref{TablePerfect} and Theorem \ref{KDF} (iv).
\end{proof}

\begin{lem}\label{a=3,4,5}
The multiset $L=\{1^a,2^b,3^c,5^d\}$ has a linear realization
if $(a,b,c,d)$ is any of the following:
\begin{enumerate}
\item[\rm{(i)}] $(3,0,1,d)$ for all $d\geq 0$;
\item[\rm{(ii)}] $(3,1,0,d)$ if, and only if, $d\neq 5k+8$;
\item[\rm{(iii)}] $(4,0,0,d)$ if, and only if, $d\neq 5k+8$;
\item[\rm{(iv)}] $(4,1,0,5k+3)$, $(5,0,0,5k+3)$ for all $k\geq 0$.
\end{enumerate}
\end{lem}

\begin{proof}
One can check that a linear realization of
$\{1^3,2,5^{(5k+8)}\}$ and of $\{1^4,5^{(5k+8)}\}$
does not exist, while a $r\{1^3,2,5^3\}$
and a $r\{1^4,5^3\}$
can be easily obtained.
All other cases can be found in \cite{PP}.
\end{proof}

Now we can prove Theorem \ref{MainResult}.

\begin{proof}[Proof of Theorem \emph{\ref{MainResult}}]
 We split the proof into some subcases, depending on the value of $a$.

1) Let $a=1$. Observe that by Remark \ref{RemA>a}, it suffices to consider the lists
$\{1,2^b,3^c,5^d\}$, where $(b,c,d)$ is either
one of the
exceptions listed in Theorem \ref{lin a=0} or it is not covered by the theorem, i.e. with
$b\leq 2$ or $c=0$.
Propositions \ref{(1,0,c,d)}, \ref{(1,1,c,d)} and \ref{(1,2,c,d)} deal with the cases
$b=0,1,2$, respectively.
The case  $L=\{1,2^b,5^d\}$ is considered in Proposition \ref{(1,b,0,d)}.
Also, for $b\geq 8$, we have $r\{1,2^b,3,5\}=R\{2^4,3\}+\{1,2^{(b-4)},5\}$, so the result
follows from Theorem \ref{KDF} (iii) and
Proposition \ref{(1,b,0,d)} (the case $b=7$ can be solved by a direct construction).
Finally, a linear realization of $\{1
,2^3,3,5^{(5k+3)}\}$
is described in \cite{PP}.

2) Let $a=2$. In view of Remark \ref{RemA>a}, we are left to consider the
exceptions of the statement for $a=1$.
By Remark \ref{a+b+c} and Lemma \ref{a=2}, $L$ is linearly realizable if and only if,
$L\neq \{1^2,5^d\},
\{1^2,3,5^d\}, \{1^2,2,5^d\}$.

3) Let $a=3$. In view of Remark \ref{RemA>a}, we have to consider
only the lists $\{1^3,5^d\}$, $\{1^3,3^1,5^d\}$, $\{1^3,2^1,5^d\}$.
By Remark \ref{a+b+c} and Lemma \ref{a=3,4,5} (i) and (ii), $L$ is linearly
realizable if and
only if $L\neq \{1^3,5^d\}, \{1^3,2,5^{(5k+8)}\}$.

4) Let $a=4$. In view of Remark \ref{RemA>a}, we have to consider
only the lists $\{1^4,5^d\}$, $\{1^4,2,5^{(5k+8)}\}$.
By Lemma \ref{a=3,4,5} (iii) and (iv) $L$ is linearly
realizable if and
only if $L\neq \{1^4,5^{(5k+8)}\}$.

5) Let $a=5$. By the previous analysis, we have to consider only the list
$L=\{1^5,5^{(5k+8)}\}$.
However, $L$ has a linear realization by Lemma \ref{a=3,4,5} (iv).

Now Remark \ref{RemA>a} implies that for $a\geq 6$ the list $L$ has always a linear
realization.
 \end{proof}

\section{Proof of Theorem \ref{Bur-Rosa} }

Let $L$ be a list whose integers are taken from $\{1,2,3,5\}$.
By \cite[Theorem 4.1]{HR} we may assume that the list $L$ contains at least three distinct
elements and
by \cite[Theorem 2.5]{CDF} we may also assume that $L$ contains $5$ at least
once. Clearly this implies that
$|L|\geq 9$.
One can check that if $L$ is such a list, then  $\BHR(L)$
reduces to
``There exists a Hamiltonian path $H$ in $K_v$ such that $\ell(H)=L=\{1^a,2^b,3^c,5^d\}$
if and only if $5$ does not divide $v$ when $a+b+c<4$.''\\
So, we prove that in all these cases there exists a cyclic realization $cL$ of $L$.
Let $L=\{2^b,3^c,5^d\}$ with $b,c,d\geq 1$.
If $b=1$ by Remark \ref{cyclin}  it suffices to find a
cyclic realization of $L$ in all the exceptional cases
not covered by Proposition \ref{(0,1,c,d)}. We list them below :
\begin{footnotesize}% d=1
$$\left.\begin{array}{l}
c\{2,3^{(3t+7)},5\}= [0,\ldots,3t+9, 1, 3t+8,\ldots, 2, 3t+7, \ldots,4],\\
c\{2,3^{(3t+8)},5\}=  [0, 3t+8, 3t+6,\ldots,3, 3t+9, 1,\ldots,3t+10, 2, \ldots,3t+5 ],\\
c\{2,3^{(3t+9)},5\}= [0,\ldots,3t+9, 3t+4,\ldots,1, 3t+10, 3t+7, 3t+5, \ldots,2, 3t+11,
3t+8 ],\\[4pt]

c\{2,3^{(3t+6)},5^2\}= [0,\ldots,3t+3, 3t+8, \ldots,2, 3t+9, 3t+7, \ldots, 1, 3t+6 ],\\
c\{2,3^{(3t+7)},5^2\}= [0,\ldots,3t+9, 1, 4, 3t+10, 3t+8, \ldots, 2, 7,\ldots,3t+7 ],\\
c\{2,3^{(3t+8)},5^2\}= [0,\ldots,3t+9, 3t+11, 2,\ldots,3t+5, 3t+10, \ldots, 1, 3t+8
],\\[4pt]
c\{2,3^{(3t+5)},5^3\}= [0, 5,\ldots,3t+8, 3,\ldots,3t+6, 1,\ldots,3t+7, 3t+9, 2 ],\\
c\{2,3^{(3t+6)},5^3\}= [0, 3t+8, 3t+6, \ldots,3, 3t+9, 3t+4,\ldots,1, 3t+7, 3t+10,
2,\ldots,3t+ 5 ],\\
c\{2,3^{(3t+7)},5^3\}= [0, \ldots,3t+9, 2,\ldots,3t+2, 3t+7, 3t+10, 3t+5, 3t+8, 3t+11, 1,
\ldots,3t+4],\\[4pt]
%\end{array}\right.$$
%\end{footnotesize}
%
%\begin{footnotesize}% d=1
%$$\left.\begin{array}{l}
c\{2,3^{(3t+4)},5^4\}=[ 0, 5,\ldots,3t+8, 3, 1, 6, 9,\ldots,3t+9, 2, 3t+7,\ldots, 4 ],\\
c\{2,3^{(3t+5)},5^5\}= [0, 3,\ldots,3t+3, 3t+8, \ldots,5, 10,\ldots,3t+10, 1, 4, 7, 2,
3t+9, 3t+11, 3t+6 ],\\
c\{2,3^{(3t+4)},5^6\}= [0, 3t+9,\ldots,9, 4, 7, 2, 5, 10,\ldots,3t+10, 3, 8,\ldots,3t+11,
1, 6 ].
\end{array}\right.$$
\end{footnotesize}

A linear realization for the other congruence classes of $c$ modulo $3$ when $d=4,5,6$ can be found in the proof of
Proposition \ref{(0,1,c,d)}, so we do not need to find a cyclic realization in these cases.
\begin{footnotesize}% c=1
$$\left.\begin{array}{rl}
c\{2,3,5^{(5k+8)}\}=\hspace{-0.3cm}&[0,\ldots,5k+10,4,\ldots,5k+9,5k+6,\ldots,1,3,\ldots,5k+8,2,\ldots,
5k+7],\\
c\{2,3,5^{(5k+9)}\}=\hspace{-0.3cm}&[0,5k+7,\ldots,2,5k+9,\ldots,4,5k+11,\ldots,1,5k+10,\ldots,5,3,\ldots,
5k+8],\\
c\{2,3,5^{(5k+10)}\}=\hspace{-0.3cm}&[0,\ldots,5k+5,5k+8,5k+10,2,\ldots,5k+12,4,\ldots,5k+9,1,\ldots,5k+11
,\\
&3,\ldots,5k+3],\\
c\{2,3,5^{(5k+11)}\}=\hspace{-0.3cm}&[0,5k+9,5k+12,\ldots,2,5k+11,\ldots,1,5k+10,\ldots,5,
3,\ldots,5k+13,\\
&4,\ldots,5k+4], \\[4pt]
% \end{array}\right.$$
% \end{footnotesize}
%
% \begin{footnotesize}% c=1
% $$\left.\begin{array}{rl}
c\{2,3^2,5^{(5k+7)}\}=\hspace{-0.3cm}&[0,\ldots,5k+5,5k+8,5k+10,4,\ldots,5k+9,3,\ldots,5k+3,5k+6,\ldots,1,\\
& 5k+7,\ldots,2], \\
c\{2,3^2,5^{(5k+8)}\}
=\hspace{-0.3cm}&[0,\ldots,5k+5,5k+8,\ldots,3,1,\ldots,5k+11,4,\ldots,5k+9,2,\ldots,
5k+7,5k+10],\\

c\{2,3^2,5^{(5k+9)}\}=\hspace{-0.3cm}&[0,\ldots,5k+5,5k+8,5k+10,2,\ldots,5k+12,4,\ldots,5k+9,1,\ldots,5k+6,\\
& 5k+3,\ldots,3, 5k+11 ] , \\
c\{2,3^2,5^{(5k+10)}\}
=\hspace{-0.3cm}&[0,\ldots,5k+5,5k+7,\ldots,2,5k+11,\ldots,1,5k+10,5k+13,\ldots,3,5k+12,\\
&5k+9,\ldots,4],\\[4pt]

c\{2,3^3, 5^{(5k+6)} \}
=\hspace{-0.3cm}&[0,\ldots,5k+5,5k+8,5k+10,2,\ldots,5k+7,1,\ldots,5k+6,5k+3,\ldots,3,\\
&5k+9,\ldots,4],\\
c\{ 2,3^3, 5^{(5k+7)} \}
=\hspace{-0.3cm}&[0,\ldots,5k+5,5k+8,1,\ldots,5k+6,5k+3,\ldots,3,5k+10,5k+7,\ldots,2,\\
&5k+9,5k+11,4,\ldots,5k+4],\\
c\{2,3^3,5^{(5k+8)}\}=\hspace{-0.3cm}&[0,\ldots,5k+5,5k+8,5k+10,2,\ldots,5k+12,5k+9,\ldots,4,1,\ldots,
5k+11,\\
&3,\ldots,5k+3 ] , \\
c\{2,3^3,5^{(5k+9)}\}=\hspace{-0.3cm}&[0,\ldots,5k+10,1,4,\ldots,5k+9,5k+11,\ldots,6,3,\ldots,5k+13,2,
\ldots,5k+12],\\
c\{2,3^3,5^{(5k+10)}\}=\hspace{-0.3cm}&[0,5k+10,\ldots,5,2,\ldots,5k+12,5k+14,\ldots,4,1,\ldots,5k+11,5k+8,\\
&5k+13,3,\ldots,5k+3 ],\\[4pt]

c\{2,3^4,5^{(5k+5)}\}=\hspace{-0.3cm}& [0, 3,\ldots,5k+8, 2, 5,\ldots,5k+10, 1,
5k+7,\ldots,7,4,
\ldots,5k+9, 5k+6,\ldots,6 ] .
\end{array}\right.$$
\end{footnotesize}

If $b=2$ by Proposition \ref{(0,2,c,d)} it suffices to consider the following
cyclic realizations:
\begin{footnotesize}% d=1
$$\left.\begin{array}{rl}
c\{2^2,3^{(3t+6)},5\}=\hspace{-0.3cm}& [0,2,\ldots,3t+5,3t+7,\ldots,1,3t+8,3,\ldots,3t+9]
,\\
c\{2^2,3^{(3t+7)},5\}=\hspace{-0.3cm}&
[0,3,5,\ldots,3t+8,2,3t+10,\ldots,4,6,\ldots,3t+9,1],\\
c\{2^2,3^{(3t+8)},5\}=\hspace{-0.3cm}&
[0,3t+9,\ldots,6,4,\ldots,3t+10,1,3,8,\ldots,3t+11,2,5],\\[4pt]

c\{2^2,3^{(3t+5)},5^2\}=\hspace{-0.3cm}&[0,5,2,4,\ldots,3t+7,3t+9,\ldots,3,8,\ldots,3t+8,1
] ,\\
c\{2^2,3^{(3t+6)},5^2\}=\hspace{-0.3cm}&
[0,3,3t+9,\ldots,6,1,\ldots,3t+7,3t+5,3t+8,3t+10,2,\ldots,3t+2],\\
c\{2^2,3^{(3t+7)},5^2\}=\hspace{-0.3cm}&
[0,3t+9,\ldots,6,4,1,3t+10,3,5,\ldots,3t+11,2,7,\ldots,3t+7],
\end{array}\right.$$
\end{footnotesize}

\begin{footnotesize}% d=1
$$\left.\begin{array}{rl}
c\{2^2,3,5^{(5k+7)}\}=\hspace{-0.3cm}&[0,5k+9,5k+6,\ldots,1,5k+7,\ldots,2,5k+8,\ldots,3,5,
\ldots,5k+10,\\
&4,\ldots,5k+4],\\
c\{2^2,3,5^{(5k+8)}\}=\hspace{-0.3cm}&[0,\ldots,5k+5, 5k+7,\ldots, 2,5k+ 9,\ldots, 4, 5k+11, 5k+6,
5k+8,\ldots, 3,\\
& 5k+10,1,\ldots,5k+1 ],\\
c\{2^2,3,5^{(5k+9)}\}=\hspace{-0.3cm}&[0, 5k+8,\ldots, 3, 5k+11, \ldots, 1, 5k+9, 5k+12, \ldots,7,
5,\ldots,5k+10,2,\\
&  4,\ldots,5k+4 ] ,\\
c\{2^2,3,5^{(5k+10)}\}=\hspace{-0.3cm}&[0, 5k+9,\ldots, 4, 5k+13, \ldots,8, 10,\ldots,5k+10, 1,\ldots,5k+
11, 2, 5, \\
&7,\ldots,5k+12,3 ],\\[4pt]

c\{2^2,3^2,5^{(5k+7)}\}=\hspace{-0.3cm}&
[0,5,2,4,\ldots,5k+9,5k+11,\ldots,1,5k+8,\ldots,3,5k+10,\ldots,10,\\
&7,\ldots,5k+7],\\
c\{2^2,3^2,5^{(5k+8)}\}=\hspace{-0.3cm}&[0,\ldots,5k+5, 5k+7,\ldots, 2, 5k+10, 5k+12, 4,\ldots,5k+ 9,
1,\ldots,5k+ 6,\\
& 5k+3,\ldots,3, 5k+11, 5k+8 ] ,\\
c\{2^2,3^2,5^{(5k+9)}\}=\hspace{-0.3cm}&
[0,\ldots,5k+5,5k+3,\ldots,3,5k+12,\ldots,2,5k+11,5k+8,5k+6,\ldots,1,\\
&5k+10,5k+13,4,\ldots,5k+9],\\
c\{2^2,3^3,5^{(5k+8)}\}= \hspace{-0.3cm}&
[0,\ldots,5k+10,5k+12,5k+9,\ldots,4,1,5k+13,\ldots,3,6,\ldots,5k+11,\\
&2,\ldots,5k+7].
\end{array}\right.$$
\end{footnotesize}

If $b\geq 3$, by Remark \ref{cyclin} it suffices to find cyclic realizations $cL$ for the
exceptions of Theorem \ref{lin a=0}. For all $v\geq 0$ we have
\begin{footnotesize}
$$\left.\begin{array}{l}
c\{2^{(2v+7)},3,5\}=[0,\ldots,2v+4,2v+9,\ldots,1,2v+8,
2v+6],\\
c\{2^{(2v+8)},3,5\}=[0,\ldots,2v+10,1,2v+9,3,\ldots,2v+7].
\end{array}\right.$$
\end{footnotesize}
Also, we obtain
\begin{footnotesize}
$$c\{2^3,3,5^{(5k+8)}\}=[0,5k+8,\ldots,3,5,\ldots,5k+10,
5k+12,\ldots,2,4,\ldots,5k+9 ,1,5k+11,\ldots,6].$$
\end{footnotesize}
If $a\geq 1$, by Remark \ref{cyclin} it suffices to find cyclic realizations for the exceptions
of
Theorem \ref{MainResult}.
Let $L=\{1^a,3^c,5^d\}$ with $a,c,d\geq 1$.
For all $t\geq 0$ we have
\begin{footnotesize}% (1,0,c,1), (1,0,c,2)
$$\left.\begin{array}{l}
c\{1,3^{(3t+7)},5\}=[0,\ldots,3t+6,3t+7,\ldots,4,3t+9,2,\ldots,3t+8,1],\\
c\{1,3^{(3t+8)},5\}=[0,3,6,5,\ldots,3t+8,2,3t+10,\ldots,1,3t+9,\ldots,9],\\
c\{1,3^{(3t+9)},5\}=[0,3t+9,\ldots,3,2,\ldots,3t+11,4,1,3t+10,\ldots,7],\\[4pt]

c\{1,3^{(3t+6)},5^2\}=[0,\ldots,3t+6,1,\ldots,3t+7,3t+2,\ldots,2,3t+9,3t+8,3t+5],\\
c\{1,3^{(3t+7)},5^2\}=[0,\ldots,3t+9,1,4,3t+10,\ldots,7,8,\ldots,3t+8,2,5],\\
c\{1,3^{(3t+8)},5^2\}=[0,3t+9,\ldots,3,3t+10,1,4,5,\ldots,3t+11,2,7,\ldots,3t+7].
\end{array}\right. $$
\end{footnotesize}

Furthermore, we have
\begin{footnotesize}%$(1,0,1,d)$,
$$\left.\begin{array}{rl}
c\{1,3,5^{(5k+8)}\}=\hspace{-0.3cm}&[0,\ldots,5k+10,4,\ldots,5k+9,5k+6,\ldots,1,5k+7,\ldots,2
,3,\ldots,5k+8],\\
c\{1,3,5^{(5k+9)}\}=\hspace{-0.3cm}&[0,5k+7,\ldots,2,5k+9,\ldots,4,3,5k+10,\ldots,5,8,\ldots,5k+8
,1,\ldots,5k+11],\\
c\{1,3,5^{(5k+10)}\}=\hspace{-0.3cm}&[0,\ldots,5k+5,5k+8,\ldots,3,5k+11,\ldots,1,5k+9,5k+10,2,
\ldots,5k+12,\\
&4,\ldots,5k+4],\\
c\{1,3,5^{(5k+11)}\}=\hspace{-0.3cm}&[0,\ldots,5k+5,5k+8,5k+13,4,\ldots,5k+9,5k+10,1,
\ldots,5k+11,\\
&2,\ldots,5k+12,3,\ldots,5k+3],
\\[4pt]

c\{1,3^2,5^{(5k+7)}\}=\hspace{-0.3cm}&[0,\ldots,5k+10,4,3,\ldots,5k+8,2,\ldots,5k+7,1,5k+9,\ldots,9
,6,\ldots,5k+6],\\
c\{1,3^2,5^{(5k+8)}\}=\hspace{-0.3cm}&[0,5k+7,\ldots,7,4,\ldots,5k+9,2,1,\ldots,5k+11,5k+8
,\ldots,3,5k+10,\ldots,5],\\
c\{1,3^2,5^{(5k+9)}\}=\hspace{-0.3cm}&[0,5k+8,\ldots,8,5,\ldots,5k+10,2,3,5k+11,\ldots,1,5k+9,
\ldots,4,\\
&7,\ldots,5k+12],\\
c\{1,3^2,5^{(5k+10)}\}=\hspace{-0.3cm}&[0,\ldots,5k+10,1,\ldots,5k+11,5k+8,5k+9,\ldots,4,5k+13,2,
\ldots,5k+12,\\
&3,\ldots,5k+3],
%\\[4pt]
 \end{array}\right.$$
 $$\left.\begin{array}{rl}
c\{1,3^3,5^{(5k+5)}\}=\hspace{-0.3cm}&[0,\ldots,5k+5,5k+8,\ldots,3,2,\ldots,5k+7,5k+4,
\ldots,4,5k+9,\\
&5k+6,\ldots,1],\\
c\{1,3^3,5^{(5k+6)}\}=\hspace{-0.3cm}& [0,\ldots,5k+5,5k+6,\ldots,6,3,\ldots,5k+8,2,\ldots,5k+7,
5k+10,\\
&4,\ldots,5k+9,1],\\
c\{1,3^3,5^{(5k+8)}\}=\hspace{-0.3cm}&[0,3,\ldots,5k+8,5k+5,\ldots,5,2,5k+10,5k+11,\ldots,1,5k+9,
\ldots,4,\\
&5k+12,\ldots,7],\\[4pt]

c\{1^2,3,5^{(5k+7)}\}=\hspace{-0.3cm} &[0,\ldots,5k+10,4,\ldots,5k+9,5k+6,
\ldots,1,5k+7,5k+8,\ldots,3,2,\ldots,5k+2],\\
c\{1^2,3,5^{(5k+8)}\}=\hspace{-0.3cm} &[0,5k+7,\ldots,2,1,5k+8,\ldots,3,6,\ldots,5k+11,4,\ldots,5k+9
,5k+10,\ldots,5],\\
c\{1^2,3,5^{(5k+9)}\}=\hspace{-0.3cm} &[0,\ldots,5k+5,5k+8,\ldots,3,2,5k+10,5k+11,\ldots,1,5k+9,
\ldots,4,\\
&5k+12,\ldots,7],\\
c\{1^2,3,5^{(5k+10)}\}=\hspace{-0.3cm} &[0,\ldots,5k+10,1,\ldots,5k+6,5k+9,\ldots,4,5k+13,\ldots,3,2
,5k+11,\\
&5k+12,\ldots,7].
\end{array}\right.$$
\end{footnotesize}

Let $L=\{1^a,2^b,5^d\}$ with $a,b,d\geq 1$.
We have
\begin{footnotesize}% $(1,1,0,d)$
$$\left.\begin{array}{rl}
c\{1,2,5^{(5k+8)}\}=\hspace{-0.3cm}&[0,\ldots,5k+10,4,\ldots,5k+4,5k+6,\ldots,1,5k+7,\ldots,2,5k+8,
5k+9,\\
&3,\ldots,5k+3],\\
c\{1,2,5^{(5k+9)}\}=\hspace{-0.3cm}&[0,\ldots,5k+10,3,2,\ldots,5k+7,5k+9,\ldots,4,5k+11,\ldots,1,
5k+8,\ldots,8],\\
c\{1,2,5^{(5k+10)}\}=\hspace{-0.3cm}&[0,\ldots,5k+10,5k+8,\ldots,3,5k+11,\ldots,1,2,\ldots,5k+12,4,
\ldots,5k+9],\\
c\{1,2,5^{(5k+11)}\}=\hspace{-0.3cm}&[0,\ldots,5k+10,1,\ldots,5k+11,5k+9,\ldots,4,5k+13,\ldots,3,2,
\ldots,5k+12],\\ [4pt]
\end{array}\right.$$
\end{footnotesize}

\begin{footnotesize}% $(1,1,0,d)$
$$\left.\begin{array}{rl}
c\{1,2^2,5^{(5k+7)}\}=\hspace{-0.3cm}&[0,5k+6,\ldots,1,5k+7,\ldots,7,5,\ldots,5k+10,4,2,5k+8,\ldots,8,
\\
&9,\ldots,5k+9,3],\\
c\{1,2^2,5^{(5k+8)}\}=\hspace{-0.3cm}&[0,5k+7,\ldots,7,5,\ldots,5k+10,3,\ldots,5k+8,5k+6,5k+11,4,
\ldots,5k+9,\\
&2,1,\ldots,5k+1],\\
c\{1,2^2,5^{(5k+9)}\}=\hspace{-0.3cm}&[0,\ldots,5k+10,2,\ldots,5k+12,1,\ldots,5k+11,5k+9,\ldots,4,3,
\ldots,5k+8],\\
c\{1,2^2,5^{(5k+10)}\}=\hspace{-0.3cm}&[0,\ldots,5k+10,1,\ldots,5k+6,5k+8,5k+13,4,\ldots,5k+9,5k+7,
\ldots,2,\\
&5k+11,
5k+12,3,\ldots,5k+3],\\[4pt]

c\{1,2^3,5^{(5k+7)}\}=\hspace{-0.3cm}&[0,5k+10,\ldots,5,3,\ldots,5k+8,1,\ldots,5k+6,5k+7,\ldots,2,5k+9
,5k+11,\\
&4,\ldots,5k+4],\\
c\{1,2^3,5^{(5k+9)}\}=\hspace{-0.3cm}&[0,\ldots,5k+10,1,3,5k+12,5k+13,4,\ldots,5k+9,5k+7,
\ldots,2,\\
&5k+11,\ldots,6,8,\ldots,5k+8],\\[4pt]
% \end{array}\right.$$
% \end{footnotesize}
%
% \begin{footnotesize}% $(1,1,0,d)$
% $$\left.\begin{array}{rl}
c\{1^2,2,5^{(5k+7)}\}=\hspace{-0.3cm}&[0,5k+6,\ldots,1,5k+7,\ldots,2,5k+8,
\ldots,3,5,4,\ldots,5k+9, 5k+10,\ldots,10],\\
c\{1^2,2,5^{(5k+8)}\}=\hspace{-0.3cm}&[0,\ldots,5k+10,5k+11,4,\ldots,5k+9,2,\ldots,5k+7,5k+8,\ldots
,3,1,\ldots,5k+6],\\
c\{1^2,2,5^{(5k+9)}\}=\hspace{-0.3cm}&[0,\ldots,5k+10,2,\ldots,5k+12,5k+11,3,4,\ldots,5k+9,1,\ldots
,5k+6,\\
&5k+8,\ldots,8],\\
c\{1^2,2,5^{(5k+10)}\}=\hspace{-0.3cm}&[0,5k+9,\ldots,4,5k+13,5k+12,\ldots,2,5k+11,\ldots,6,5,
\ldots,5k+10,1,\\
& 3,\ldots,5k+8],\\[4pt]

c\{1^3,2,5^{(5k+8)}\}=\hspace{-0.3cm}&[0,\ldots,5k+5,5k+4,\ldots,4,5k+12,1,\ldots,5k+11,
3,\ldots,5k+8,5k+9,\\
&5k+10,2,\ldots,5k+7].\\
\end{array}\right.$$
\end{footnotesize}

Let $L=\{1^a,2^b,3^c,5^d\}$ with $a,b,c,d\geq 1$.
We have
\begin{footnotesize}%(1,1,1,d)
$$\left.\begin{array}{rl}
c\{1,2,3,5^{(5k+7)}\}=\hspace{-0.3cm} &[0,5k+6,\ldots,1,5k+7,\ldots,2,5k+10
,\ldots,5,3,\ldots,5k+8,5k+9,\ldots,4],\\
c\{1,2,3,5^{(5k+8)}\}=\hspace{-0.3cm}&[0,\ldots,5k+10,3,\ldots,5k+8,5k+7,\ldots,2,5k+9,5k+11,\ldots
,1,\\
&4,\ldots,5k+4],\\
c\{1,2,3,5^{(5k+9)}\}=\hspace{-0.3cm}&[0,\ldots,5k+10,5k+8,\ldots,3,2,\ldots,5k+12,4,\ldots,5k+9,1,
5k+11,\ldots,6],\\
c\{1,2,3,5^{(5k+10)}\}=\hspace{-0.3cm}&[0,\ldots,5k+10,1,\ldots,5k+6,5k+9,5k+8,\ldots,3,5k+12,
\ldots,2,\\
&5k+11,5k+13,4,\ldots,5k+4 ],\\[4pt]
% \end{array}\right.$$
% $$\left.\begin{array}{rl}
c\{1,2,3^2,5^{(5k+7)}\}=\hspace{-0.3cm}&[0,\ldots,5k+10,5k+7,\ldots,2,5k+9,
\ldots,4,3,\ldots,5k+8,5k+11,\\
&1,\ldots,5k+6],\\
c\{1,2,3^2,5^{(5k+9)}\}=\hspace{-0.3cm}&[0,5k+11,2,\ldots,5k+12,5k+13,4,\ldots,5k+9,5k+6,
\ldots,1,5k+10,\ldots,5,\\
&3,\ldots,5k+8],\\

c\{1,2^2,3,5^{(5k+8)}\}=\hspace{-0.3cm}&[0,\ldots,5k+10,2,\ldots,5k+12,4,6,\ldots,5k+6,
5k+8,\ldots,8,9,\ldots,5k+9,\\
&1,5k+11,3].
\end{array}\right.$$
\end{footnotesize}

\section{Final remarks}

Clearly the strategy illustrated in Section 2 can be applied to any explicit list $L$.
So the linear realizations constructed in the previous sections
can be used to prove the conjecture for any list $L$ with underlying-set  $S$ containing any subset of $\{1,2,3,5\}$.
On the other hand, we have to point out that the larger is the size of  $S$,
the larger is also the number of particular cases one should deal with.\\
In our opinion, in order to provide a complete proof of the conjecture,
namely in order to prove that BHR$(L)$ is true \underline{for any} given list $L$,
new strategies must be found.
At this purpose in \cite{PPnew} the authors introduced a new kind of special linear realizations
which turned out to be crucial to investigate the case $S=\{1,2,t\}$, where $t$ is an arbitrary
positive integer.
At the moment we have some partial results, for instance we have proved BHR$(\{1^a,2^b,t^c\})$ when $t\leq8$ for any $a,b,c$, and when $t\equiv0\pmod4$
for any $a>1$ with $a+b\geq t-1$.
%These new results have been obtained thanks to a special class of linear realizations introduced by the authors.\\
%We point out that, in our opinion, in order to provide a complete proof of the conjecture,
%namely in order to prove that BHR$(L)$ is true \underline{for any} given list $L$,
%new strategies must be found.
%At the moment we have proved that BHR$(L)$ is true for any list $L$ whose
%underlying-set $S$ is $S=\{1,2,t\}$, whit $t=4,\ldots,8$. This partial result

\section*{Acknowledgments}
The authors thank  the anonymous referees for their helpful comments and suggestions.


\begin{thebibliography}{99}

\bibitem{BE}
D. Bryant, S. El-Zanati, Graph decompositions,
In: C. J. Colbourn and J. H. Dinitz editors, Handbook of combinatorial designs,
Second Edition, Chapman \& Hall/CRC, Boca Raton, FL, 2006, pp. 477--486.

\bibitem{BCD} M. Buratti, S. Capparelli, A. Del Fra,
\emph{Cyclic Hamiltonian cycle systems of the $\lambda$-fold
complete and cocktail party graphs},
European J. Combin. \textbf{31} (2010), 1484--1496.

\bibitem{BMnew} M. Buratti, F. Merola,
\emph{Hamiltonian cycle systems which are both cyclic and symmetric},
to appear in J. Combin. Des., doi:10.1002/jcd.21351.

\bibitem{BM} M. Buratti, F. Merola,
\emph{Dihedral hamiltonian cycle systems
of the cocktail party graph},
J. Combin. Des. \textbf{21} (2013), 1--23.

\bibitem{BP} M. Buratti, A. Pasotti, \emph{On perfect
$\Gamma$-decompositions of the complete graph}, J.
Combin. Des. \textbf{17} (2008), 197--209.

\bibitem{CDF} S. Capparelli, A. Del Fra,
\emph{Hamiltonian paths in the complete graph with edge-lengths $1,2,3$},
Electron. J. Combin. \textbf{17} (2010), $\sharp$R44.

\bibitem{DJ} J.H. Dinitz, S.R. Janiszewski.
\emph{On hamiltonian paths with prescribed edge lengths in the complete graph},
Bull. Inst. Combin. Appl. \textbf{57} (2009), 42--52.

\bibitem{GR} C. Godsil, G. Royle, Algebraic graph theory.
Graduate Texts in Mathematics. Vol 207. Springer, (2001).

%\bibitem{G} R.J. Gould, \emph{Advances on Hamiltonian problem - a survey},
%Graphs Combin. \textbf{19} (2003), 7--52.

\bibitem{HR} P. Horak, A. Rosa,
\emph{On a problem of Marco Buratti},
Electron. J. Combin. \textbf{16} (2009), $\sharp$R20.

\bibitem{PP} A. Pasotti, M.A. Pellegrini, http://www.ing.unibs.it/$\thicksim$anita.pasotti/BHRproblem.pdf

\bibitem{PPnew} A. Pasotti, M.A. Pellegrini, \emph{Some results on} BHR$(\{1^a,2^b,t^c\})$,
in preparation.

\bibitem{SS}  B. Seamone, B. Stevens,
\emph{Spanning trees with specified differences in Cayley graphs},
Discrete Math. \textbf{312} (2012), 2561–-2565.

 \bibitem{W} D. West, http://www.math.uiuc.edu/$\thicksim$west/regs/buratti.html.
%
\bibitem{Wbook} D. West, Introduction to graph theory. Prentice Hall, New Jersey (1996).


\end{thebibliography}
\end{document}